\documentclass[11pt]{amsart}
\calclayout
\input{Style.sty}
\title{Maurer--Cartan elements in symplectic cohomology from compactifications}
\author{Matthew Strom Borman, Mohamed El Alami, and Nick Sheridan}
\begin{document}
\begin{abstract}
    We prove that under certain conditions, a normal crossings compactification of a Liouville domain determines a Maurer--Cartan element for the $L_\infty$ structure on its symplectic cohomology; and deforming by this element gives the quantum cohomology of the compactification.
\end{abstract}

\maketitle

\section{Introduction}

\subsection{Geometric setup}

Let $(M,\omega)$ be a compact symplectic manifold which is positively monotone, so that 
\begin{equation}\label{eq:monotonicity}
[\omega] = 2\kappa c_1(M) \quad \text{in}\ H^2(M,\mathbb{R}), \quad \kappa>0.
\end{equation}

Let $D= (D_i)_{i=1}^N \subset M$ be an orthogonal\footnote{The assumption that $D$ is orthogonal does not cause any loss of generality by \cite{McLean-growth-rate}, as observed in \cite[Remark 2.2]{BSV}.} simple crossings divisor (we recall the definition in Section \ref{sec:scd}) such that
\begin{equation}
    2c_1(M) = \lambda_1 \pd(D_1) + \dotsi + \lambda_N \pd(D_N),
\end{equation}
where $\pd: H_{2n-2}(M,\mathbb{Q})\rightarrow H^2(M,\mathbb{Q})$ is Poincar\'e duality and $\lambda_i \in \mathbb{Q}_{>0}$ are positive rational numbers. This condition essentially says that $D$ supports (a multiple of) an effective anti-canonical divisor. 

Let $\bk$ be a field of characteristic zero, and define the $\Q$-graded $\bk$-algebra
\begin{align*}\label{eq: ring R+}
    &R^+ = \bk[q_1,\dots,q_N],\quad \deg(q_j) = \lambda_j,
\end{align*}
as well as its localization at the product $q_1\dotsi q_N$,
\begin{equation}
    R  = \bk[q_1^{\pm 1},\dots,q_N^{\pm 1}].
\end{equation}
These are the Novikov-type rings we work with. 
We define the quantum cohomology to be the $\Q$-graded $R$-module 
$$QH^*(M;R) \defeq H^*(M;\bk) \otimes_\bk R.$$

The complement $X = M \setminus D$ carries a Liouville structure which is convex and of finite type in the sense of \cite{McLean-growth-rate}. 
Several definitions of an $L_\infty$ algebra structure on symplectic cochains $SC^*(X;\bk)$ are available \cite{Pomerleano-Seidel-1,Abouzaid-Groman-Varolgunes,ES1}; we will use the definition from \cite{ES1} as it is adapted to the purposes of the present paper.

\subsection{Maurer--Cartan element}

We briefly recall the formalism of Maurer--Cartan elements in $L_\infty$ algebras, following \cite{Getzler-Lie-theory}. 

Let $(\mathfrak{g},\ell^*)$ be an $L_\infty$ algebra over $\bk$, and $(R,Q_{\ge \bullet})$ a graded filtered $\bk$-algebra. 
We define $\mathfrak{g}^R$ to be the degreewise completion of $\mathfrak{g} \otimes R$; it carries an $L_\infty$ structure by linear extension. 
Define the Maurer--Cartan set
\begin{equation}
    MC(\mathfrak{g},R) \defeq \left\{\beta \in Q_{\geq 1}\mathfrak{g}^R \ | \ \deg(\beta)=2\quad\text{and}\quad \mathcal{F}(\beta) = 0\right\},
\end{equation}
where $\mathcal{F}(\beta) \in \mathfrak{g}^R$ is the curvature term
\begin{equation}\label{eq: MC}
    \mathcal{F}(\beta) = \sum_{d=1}^{\infty} \frac{1}{d!} \ell^d(\beta,\dots,\beta).
\end{equation}
(The appearance of $d!$ in the denominator here is the reason we need to assume that $\bk$ has characteristic zero.)

Elements $\beta\in MC(\mathfrak{g},R)$ are used to define deformed $L_{\infty}$ operations,
\begin{equation}
    \ell_{\beta}^d(x_1,\dots,x_d) = \sum_{k=0}^{\infty} \frac{1}{k!}\ell^{k+d}(\beta^{\otimes k},x_1,\dots,x_d),
\end{equation}
see \cite[Proposition 4.4]{Getzler-Lie-theory}. 
Given an element $\gamma \in Q_{\ge 1} \mathfrak{g}^R$ of degree $1$, we may define a vector field $\alpha \mapsto \ell^1_\alpha(\gamma)$ which is tangent to $MC(\mathfrak{g},R)$; Maurer--Cartan elements which lie on a common flowline of such a vector field are said to be gauge equivalent, and give rise to quasi-isomorphic deformed $L_\infty$ algebras.

The main construction of this paper is:

\begin{construction}\label{const}
    We construct a Maurer--Cartan element 
    $$\beta \in MC(SC^*(X;\bk),R^+).$$
    
\end{construction}

The linear term of the Maurer--Cartan element gives a well-defined cohomology class
$$[D\beta(0)] \in SH^*(X;\bk) \otimes \mathfrak{m}/\mathfrak{m}^2$$
of degree $2$, where $\mathfrak{m} \subset R^+$ is the ideal generated by the $q_i$. This takes the form
$$[D\beta(0)] = \sum_{j=1}^N q_j \cdot \beta_j,$$
where $\beta_j \in SH^{2-\deg(q_j)}(X;\bk)$. 
Under assumptions where sphere bubbling is ruled out (e.g., if there is no spherical class $A \in H_2(M)$ with $\omega(A) > 0$, $A \cdot D_j = 1$, and $A \cdot D_k = 0$ for $k \neq j$), we expect that
$$\beta_j = PSS^{log}(\gamma_j)$$
where $PSS^{log}$ is the `log PSS map' constructed in \cite{Ganatra-Pomerleano-top-limit,Ganatra-Pomerleano-log-PSS}, and $\gamma_j$ is the class in the log cohomology of $(M,D)$ corresponding to the divisor $D_j$, as in \cite[Section 6]{GHHPS}. 
Indeed, it should be possible to prove this with the techniques used to prove \cite[Proposition 6.11]{GHHPS}.

\subsection{Deforming by the Maurer--Cartan element}

Our main result says that the deformation of symplectic cohomology by the Maurer--Cartan element $\beta$ recovers the quantum cohomology, but this only holds under the following:

\begin{hypothesis}[Hypothesis A of \cite{BSV}]\label{hyp}
    For all $j=1,\dots,N$, we have $\lambda_j \leq 2$.
\end{hypothesis}

We then have:

\begin{theorem}\label{thm}
    Under Hypothesis \ref{hyp}, we have an isomorphism of graded $R$-modules
    $$QH^*(M;R) \cong H^*\left(SC^*(X;\bk) \wotimes_\bk R,\ell^1_{\beta}\right).$$
\end{theorem}

This proves \cite[Conjecture 1.25]{BSV}. 
Note that we have abused notation slightly by continuing to use the same notation $\beta$ for the image of the Maurer--Cartan element under the inclusion $$MC(SC^*(X;\bk),R^+) \hookrightarrow MC(SC^*(X;\bk),R).$$

We remark that Pomerleano--Seidel have proved an analogue of Theorem \ref{thm} in \cite[Corollary 1.2.2]{Pomerleano-Seidel}, in the case that $D$ is a smooth anticanonical divisor (i.e., $N=1$ and $\lambda_1 = 2$). Their construction of the Maurer--Cartan element $\beta$ is quite different from ours, but it can be shown relatively easily to be gauge equivalent to ours: indeed, grading considerations show that the quadratic and higher terms in the Maurer--Cartan element vanish on the chain level, as a result of which the gauge equivalence class is completely determined by the cohomology class of its linear term $[D\beta(0)]$. 

In fact, Pomerleano--Seidel prove a more precise result: they identify the cohomology of $(SC^*(X;\bk) \wotimes R^+,\ell^1_\beta)$ \cite[Theorem 1.2.1]{Pomerleano-Seidel}. 
It is an interesting question to determine the analogous result in the context of normal-crossings $D$ which we consider here, cf. \cite[Section 1.6.1]{BSV}. 
They also allow $\bk$ to be an arbitrary commutative ring, whereas we require it to be a field of characteristic zero, due to our use of Maurer--Cartan deformation theory.

\subsection{Conjectures}\label{sec:conj}

We conjecture a generalization of Theorem \ref{thm} to the case when Hypothesis \ref{hyp} does not apply:

\begin{conjecture}[Cf. Remark 1.21 of \cite{BSV}]\label{conj:no hyp}
    We have an isomorphism of graded $R$-modules
    \begin{equation}
        \label{eq:thm gen}
    SH^*_M(\mathbb{L};R) \cong H^*\left(SC^*(X;\bk) \wotimes_\bk R, \ell^1_\beta\right),
    \end{equation}
    where $\mathbb{L}$ is the skeleton of $X$ and $SH^*_M$ denotes the relative symplectic cohomology.
\end{conjecture}

Note that Conjecture \ref{conj:no hyp} recovers Theorem \ref{thm} when Hypothesis \ref{hyp} applies, by \cite[Theorem D]{BSV}. 
On the other hand, it is strictly more general: for example, it holds in cases where both sides of \eqref{eq:thm gen} are $0$, such as when $M = \mathbb{CP}^n$ and $D$ is a hyperplane (cf. \cite[Remark 1.5]{BSV}). 
Another interesting case of Conjecture \ref{conj:no hyp} to consider, is the case that $M = \mathbb{CP}^1 \times \mathbb{CP}^1$ and $D$ is a $(1,1)$ hypersurface. 
Following work of Eliashberg and Polterovich \cite{EP}, we expect that both sides of \eqref{eq:thm gen} are non-zero, yet also not isomorphic to $QH^*(M;R)$; indeed, we expect they should be isomorphic to the $0$-generalized eigenspace of quantum cup product by $D$.

In a different direction, it is natural to conjecture that the isomorphism of Theorem \ref{thm} (or indeed Conjecture \ref{conj:no hyp}) respects certain natural algebraic structures, cf. \cite{Fabert}. 
For example, Pomerleano--Seidel prove that (when $D$ is smooth and anticanonical) their isomorphism respects the quantum connection \cite[Proposition 1.2.7]{Pomerleano-Seidel}.

Our next conjecture concerns the relationship of our construction with Seidel's relative Fukaya category $\mathcal{F}(M,D)$ \cite{Seidel-rel-fuk}. 

\begin{definition}
We define $\mathcal{W}(M,D)$ to be the curved filtered $R^+$-linear $A_\infty$ deformation of the wrapped Fukaya category $\mathcal{W}(X;\bk)$ corresponding to the pushforward of the Maurer--Cartan element 
$$\mathrm{CO}_*\beta \defeq \sum_{d \ge 1} \frac{1}{d!} \mathrm{CO}^d(\beta,\ldots,\beta) \in MC(CC^*(\mathcal{W}(X;\bk),R^+)$$ 
along the $L_\infty$ closed--open map $\mathrm{CO}:SC^*(X;\bk) \to CC^*(\mathcal{W}(X;\bk))$ constructed in \cite{ES1} (see also \cite{Pomerleano-Seidel-1}). 
\end{definition}

\begin{conjecture}\label{conj:relfuk}
    The full subcategory $\mathcal{W}(M,D)$ whose objects are compact Lagrangians is filtered quasi-equivalent to Seidel's relative Fukaya category $\mathcal{F}(M,D)$ \cite{Seidel-rel-fuk}. 
\end{conjecture}

\begin{remark}
    Tonkonog's \cite[Theorem 6.5]{Tonkonog} constitutes a verification of part of Conjecture \ref{conj:relfuk}, namely the part with $d=1$ and length-zero Hochschild cochains.
\end{remark}

\begin{remark}
    As pointed out to us by Ian Zemke, the `Bordered $HF^-$ algebra' constructed in \cite{LOT} should be interpreted as a subcategory of $\mathcal{W}(M,D)$, for certain $(M,D)$, whose objects can't be interpreted as objects of $\mathcal{F}(M,D)$, because they are Lagrangians with non-empty boundary on $D$. 
    (We note that in this case $M$ need not be positively monotone, so strictly speaking the construction in this paper need not apply; however we expect that an analogous construction can be made more broadly.)
\end{remark}

\begin{remark}
    The `deformed gentle algebra' introduced in \cite{Bocklandt_vdK} should also be interpreted as a subcategory of $\mathcal{W}(M,D)$, whose objects can't be interpreted as objects of $\mathcal{F}(M,D)$.
\end{remark}

Finally, we conjecture that all of our results and conjectures have analogues in the case that $M$ is Calabi--Yau, rather than positively monotone (with no Hypothesis \ref{hyp} required in this case), cf. \cite{McLean-birational-CY,Sun-SH}.

We expect that all of the conjectures discussed in this section are well within reach of the techniques developed in this paper (combined with the techniques of \cite{BSV} in the case of Conjecture \ref{conj:no hyp}).

\subsection{Implications for mirror symmetry in the log Calabi--Yau case}

Following \cite{Tonkonog} and \cite[Section 1.6]{BSV}, we explain how our main result, together with the conjectures of the previous section, give an alternative perspective on Auroux' work on mirror symmetry in the log Calabi--Yau case \cite{Auroux-Tdual}. 
Thus, we assume that $D$ is anticanonical, meaning $\lambda_j = 2$ for all $j$. 

Suppose that $Y$ is a smooth $\Bbbk$-variety mirror to $X$, in the sense that we have a quasi-equivalence
$$\mathcal{W}(X;\bk) \simeq \mathrm{DCoh}(Y).$$
Hacking and Keating have proved such a result for surfaces \cite{HackingKeating}; in higher dimensions, one may expect it to be an output of the `intrinsic mirror symmetry' programme, see \cite{GS-intrinsic,Pomerleano-intrinsic}.

The mirror quasi-equivalence induces $L_\infty$ quasi-isomorphisms
\begin{equation}
    \label{eq:MS-qisos}
SC^*(X;\bk) \xrightarrow{\mathrm{CO}} CC^*(\mathcal{W}(X;\bk)) \simeq CC^*(\mathrm{DCoh}(Y)) \simeq C^*(Y,\wedge^* TY).
\end{equation}
Here $C^*(Y,\wedge^* TY)$ is the DG Lie algebra of derived global sections of the sheaf $(\wedge^* TY,0,[-,-])$ of DG Lie algebras with vanishing differential, and Schouten--Nijenhuis bracket (see e.g. \cite[Appendix B]{VdB} for more on the general construction of a structure of DG Lie algebra on a chain complex underlying derived global sections of a sheaf of DG Lie algebras; in the complex case, a specific model using Dolbeault cohomology was used in \cite{Barannikov-Kontsevich}). 
The main step in the construction of the right-hand quasi-equivalence is Konstevich's Formality Theorem \cite{Kontsevich-DQ,Kontsevich-DQ-alg}. 

Because $SC^*(X;\bk)$ is concentrated in non-negative degrees, and the generators $q_j$ of $R^+$ have degree $2$, the Maurer--Cartan element $\beta \in MC(SC^*(X;\bk),R^+)$ has only linear terms, and its gauge equivalence class is completely determined by 
$$D\beta(0) = \sum_j q_j \beta_j,$$
where $\beta_j \in SH^0(X;\bk)$. 
Under \eqref{eq:MS-qisos}, the Maurer--Cartan element $\beta$ corresponds to a Maurer--Cartan element $w \in MC(C^*(Y,\wedge^* TY),R^+)$, which is similarly determined by the linear terms
$$Dw(0) = \sum_j q_j w_j$$
where $w_j \in H^0(\wedge^0 TY) = \mathcal{O}(Y)$. 
In other words, the gauge equivalence class of the Maurer--Cartan element is determined by a `multi-potential' $(w_1,\ldots,w_N):Y \to \mathbb{A}^n$, in the terminology of \cite{Lee-multip}.

The Maurer--Cartan element $w$ determines an $A_\infty$ deformation of $\mathrm{DCoh}(Y)$ over $R^+$, which we denote by $\mathrm{DCoh}(Y,w)$; it is immediate that we have a filtered quasi-equivalence of $R^+$-linear curved $A_\infty$ categories
$$\mathcal{W}(M,D) \simeq \mathrm{DCoh}(Y,w).$$
In order to get equivalences of uncurved $A_\infty$ categories, we pass to categories of weak bounding cochains in the sense of \cite{FOOO} (see also the exposition in \cite[Section 4]{SheridanFano}). 
We could do this over $R^+$, but in order to get a more familiar mirror symmetry statement, we first pass to the single-variable Novikov ring $R^+_q = \bk[q]$ with $|q|=2$, via the graded algebra homomorphism
\begin{eqnarray*}
    R^+ & \to &R^+_q \\
    q_j &\mapsto &q \quad \text{for all $j$.}
\end{eqnarray*}
We denote the resulting Maurer--Cartan elements by $\beta_q = q \sum \beta_j$, and $w_q = qW$ with $W= \sum w_j$. 
We obtain a filtered quasi-equivalence of $R^+_q$-linear curved $A_\infty$ categories
$$\mathcal{W}(M,D)_q \simeq \mathrm{DCoh}(Y,w_q).$$
This induces a filtered quasi-equivalence of the corresponding uncurved categories of weak bounding cochains with curvature $\lambda q$:
$$\mathcal{W}(M,D)^{wbc}_{\lambda q} \simeq \mathrm{DCoh}(Y,w_q)^{wbc}_{\lambda q}$$
for all $\lambda \in \bk$ (see \cite[Theorem 2.12]{perutz2022constructing} for the case $\lambda = 0$). Setting $q=1$, we obtain quasi-equivalences of $\Z/2$-graded $\bk$-linear $A_\infty$ categories, which we denote by
$$\mathcal{W}(M,D)^{wbc}_\lambda \simeq \mathrm{DCoh}(Y,W)^{wbc}_\lambda.$$

We expect that $\mathrm{DCoh}(Y,W)^{wbc}_\lambda$ coincides with the category of matrix factorizations $MF(Y,W-\lambda)$ defined in \cite{Orlov-nonaffine,LinPomerleano} (which is equivalent to Orlov's triangulated category of singularities $\mathrm{DSing}(Y_\lambda)$ \cite{Orlov-Dsing}, where $Y_\lambda$ is the fibre of $W$ over $\lambda$, by \cite[Theorem 3.5]{Orlov-nonaffine}). 
Roughly, we expect that deforming by the Maurer--Cartan element $qW$ has the simple effect of adding a curvature term $qW \cdot \mathrm{id}$ to each object. 
A matrix factorization $$E_0 \xrightarrow{d_0} E_1 \xrightarrow{d_1} E_0,\qquad \text{satisfying}\qquad d_0d_1=(W-\lambda)\cdot \id_{E_1} \quad \text{and} \quad d_1d_0=(W - \lambda) \cdot \mathrm{id}_{E_0},$$ can then be considered as a complex of coherent sheaves $$E_0[-1] \xrightarrow{d_0} E_1,$$ equipped with a weak bounding cochain $qd_1$, where the only non-vanishing terms in the Maurer--Cartan equation
$$\mathfrak{m}^0 + \mathfrak{m}^1(qd_1) + \mathfrak{m}^2(qd_1,qd_1) + \ldots = q\lambda \cdot \mathrm{id}$$
are $\mathfrak{m}^0 = qW \cdot \mathrm{id}$ and $\mathfrak{m}^1(qd_1) = q(d_0d_1 + d_1d_0) = q(W-\lambda) \cdot \id_{E_0[-1] \oplus E_1}$.

Assuming this holds, we obtain a quasi-equivalence
$$\mathcal{W}(M,D)^{wbc}_\lambda \simeq MF(W-\lambda)$$
for all $\lambda$, which is the usual statement of homological mirror symmetry, albeit with an enlarged Fukaya category. 
Of course, if the subcategory $\mathcal{F}(M,D)^{wbc}_\lambda \subset \mathcal{W}(M,D)^{wbc}_\lambda$ split-generates, we may replace $\mathcal{W}$ with the more familiar $\mathcal{F}$.

\begin{example}
    Let $M$ be a smooth Fano toric variety, and $D$ its toric boundary divisor. 
    Then the completion of $X = M \setminus D$ is $T^*T^n$, which is mirror to $Y = \mathbb{G}_m^n$. 
    It follows from \cite[Theorem 6.5]{Tonkonog} and the computation of disc potentials for monotone torus fibres in Fano toric varieties \cite{Cho-Oh} that $w_j = z^{v_j} \in \mathcal{O}(Y) = \bk[z_1^{\pm 1},\ldots,z_n^{\pm 1}]$ where $v_j$ is the ray of the fan of $M$ corresponding to component $D_j$ of $D$. 
    Thus $W = \sum_j w_j$ is the usual Hori--Vara mirror potential, so we get the expected homological mirror symmetry result for Fano toric varieties; note that this is a case where we expect $\mathcal{F}(M,D)^{wbc}_\lambda \subset \mathcal{W}(M,D)^{wbc}_\lambda$ to split-generate, as each component of the Fukaya category is split-generated by some local system on the monotone torus fibre by \cite[Corollary 1.3.1]{Evans-Lekili}.
\end{example}

In a slightly different direction, we may deform the quasi-isomorphism \eqref{eq:MS-qisos} by the respective Maurer--Cartan elements, tensor with $R_q := \bk[q^{\pm 1}]$, then take cohomology. The left-hand side becomes $QH^*(M;R_q)$ by Theorem \ref{thm}, while the right-hand side becomes $\mathcal{O}(Z^h) \otimes_\bk R_q$, where $\mathcal{O}(Z^h)$ is the algebra of functions on the derived critical locus of $W$, as argued in \cite[Section 1.6.4]{BSV}. 
In the case that the critical locus of $W$ is isolated, this is the ordinary algebra of functions on the critical locus, which decomposes as the direct sum of Jacobian algebras of the critical points of $W$. 
This gives us the familiar closed-string mirror isomorphism between quantum cohomology of $M$ and the Jacobian algebra of $W$, on the level of vector spaces -- in order to upgrade it to respect more algebraic structure, one would need to correspondingly upgrade Theorem \ref{thm} (cf. \cite{Pomerleano-Seidel-1}), as well as the Formality Theorem (cf. \cite{Tamarkin} and \cite[Section 7]{Preygelthesis}).

\subsection{Outline}

Let us first motivate Theorem \ref{thm}.  
It is shown in \cite{BSV} how to compute the quantum cohomology in terms of Hamiltonian Floer theory of a sequence of Hamiltonians which become infinite along the divisor; and that one may safely `ignore' the orbits which are not contained in $X$ for this computation. 
Each relevant Floer trajectory, then, passes between orbits disjoint from $D$. 
We now imagine `stretching the neck' along $\partial X$, in the sense of symplectic field theory \cite{sft}. 
Floer trajectories will degenerate into components contained in $X$, which can be identified with $L_\infty$ operations; and `caps' lying outside $X$, whose count should define the Maurer--Cartan element $\beta$. 
Thus we expect neck-stretching to give a chain homotopy between the Floer differential (which counts Floer trajectories), and the deformed differential $\ell^1_\beta$.

Rather than implement this approach, we use a different approach which substitutes algebra for neck-stretching analysis. 
Recall that the symplectic cochains on $X$ are defined in \cite{ES1} as the telescope complex of an increasing sequence of Hamiltonians:
$$SC^*(X;\bk) = \bigoplus_{n=1}^\infty CF^*(X,H_n;\bk)[t],$$
where $t$ is a formal variable satisfying $t^2 = 0$. 
We enlarge this telescope complex to include formal generators $x^{D_j}$ for each component $D_j$ of the divisor, which we think of as carrying weight $n=0$:
$$\mathfrak{g} = \bigoplus_{j=1}^N \bk[t] \cdot x^{D_j} \oplus \bigoplus_{n=1}^\infty CF^*(X,H_n;\bk)[t].$$
We extend the $L_\infty$ structure to this enlarged complex by counting pseudoholomorphic curves $u$ in $M$, where weight-zero inputs $x^{D_j}$ or $t \cdot x^{D_j}$ correspond bijectively to intersection points of $u$ with $D_j$. 

\begin{remark}
In fact, we are not able to prove that the extended `$L_\infty$ operations' satisfy the $L_\infty$ relations, due to breakings which may occur along orbits lying outside of $X$. 
Instead, we show that they satisfy the $L_\infty$ operations when all inputs are drawn from a certain subspace, and show that this suffices for our purposes. 
Although this is the most technically involved part of the paper, it is not conceptually enlightening so we will elide it for the remainder of this section.
\end{remark}

The extended $L_\infty$ products are defined to vanish if all inputs are formal generators $x^{D_j}$; this is valid because when the corresponding marked points bubble off in our moduli spaces they form a sphere bubble, which we may arrange to appear in codimension $2$ rather than codimension $1$ as with the other boundary components contributing to the $L_\infty$ relations. 
As a result, we have a Maurer--Cartan element
$$\alpha_0 = \sum_{j=1}^N q_j \cdot x^{D_j} \in MC(\mathfrak{g},R^+)$$
(it satisfies the Maurer--Cartan equation \eqref{eq: MC} because every term vanishes).
From $\alpha_0$, we construct a one-parameter family of gauge-equivalent Maurer--Cartan elements $\alpha_\tau$ by flowing along the vector field determined by $\gamma = t \cdot \alpha_0$. 
The weight-zero part of $\alpha_\tau$ is $(1-\tau) \cdot \alpha_0$, and in particular, $\alpha_1 \in MC(SC^*(X;\bk),R^+)$. 
The Maurer--Cartan element $\beta$ from Construction \ref{const} is defined to be $\alpha_1$.

We now outline the proof of Theorem \ref{thm}.
We define a telescope complex 
$$QC^*(M;R) \defeq \bigoplus_{n \ge 1} CF^*(M,H_n)[t]$$
associated to the sequence of Hamiltonians $H_n$, extended to $M$. 
We have an isomorphism
$$QH^*(M;R) \cong H^*(QC^*(M;R),d^{QC}),$$
simply because each chain complex in the directed system computes $QH^*(M;R)$ by the PSS isomorphism, and the continuation maps between them are isomorphisms.

In \cite{BSV}, a map 
\begin{align*}
    \Phi: SC^*(X;\bk) \otimes R & \to QC^*(M;R)
\end{align*}
is constructed, sending
$$    \gamma \mapsto [\gamma,u] \otimes q_1^{u \cdot D_1} \ldots q_N^{u \cdot D_N}$$
for an arbitrary cap $u$ for $\gamma$. 
We arrange that this $\Phi$ is a chain map:
\begin{equation}
    \label{eq:Phi ch map}
    \Phi(\ell^1_{\alpha_0}(\gamma)) = d^{QC}(\Phi(\gamma)),
\end{equation}
because the differentials on the two sides of the equation count isomorphic moduli spaces. 

\begin{remark}
In order for the infinite sum defining the differential $\ell^1_{\alpha_0}$ to be defined in general, and for our subsequent application of standard arguments from Maurer--Cartan theory, we need our complexes to be degreewise complete, i.e. to work with $SC^*(X;\bk) \wotimes R$ rather than $SC^*(X;\bk) \otimes R$. 
Under Hypothesis \ref{hyp} however, $SC^*(X;\bk)$ is concentrated in non-negative degrees, which implies that the degreewise completion has no effect: $SC^*(X;\bk) \wotimes R = SC^*(X;\bk) \otimes R$.
\end{remark}

Let us now justify \eqref{eq:Phi ch map}. 
On the RHS, the differential counts Floer trajectories and continuation maps $u$ in $M$ between orbits $\gamma_0,\gamma_1$ in $X$; while on the LHS, we arrange that the terms
$$\frac{1}{d!} \ell^{1+d}\left(\gamma_1,q_{j_1}x^{D_{j_1}},\ldots,q_{j_d}x^{D_{j_d}}\right)$$
count precisely the same curves $u$ with the same weights.
Thus (up to a harmless passage to a subquotient complex to remove orbits outside of $X$, which we elide for the purposes of this section), we obtain a chain isomorphism
\begin{equation}
    \label{eq:Phi}
\Phi:\left(SC^*(X;\bk) \otimes R,\ell^1_{\alpha_0}\right) \xrightarrow{\sim} \left(QC^*(M;R),d^{QC}\right).
\end{equation}

We also have that the inclusions
\begin{equation}
\label{eq:SC in g qiso}
    \left(SC^*(X;\bk) \otimes R,\ell^1_{\alpha_\tau}\right) \hookrightarrow \left(\mathfrak{g} \otimes R,\ell^1_{\alpha_\tau}\right)
\end{equation}
are quasi-isomorphisms for all $\tau$, as the quotient complexes are manifestly acyclic.

On the other hand, the Maurer--Cartan elements $\alpha_0$ and $\alpha_1 = \beta$ determine quasi-isomorphic deformations
\begin{equation}
    \label{eq:gauge eq qiso}
    \left(\mathfrak{g} \otimes R,\ell^1_{\alpha_0}\right) \simeq \left(\mathfrak{g} \otimes R,\ell^1_{\beta}\right)
\end{equation}
due to the fact they are gauge equivalent, by a standard argument from Maurer--Cartan theory.
Theorem \ref{thm} now follows by composing the isomorphisms \eqref{eq:Phi}, \eqref{eq:SC in g qiso} for $\tau = 0$, \eqref{eq:gauge eq qiso}, and \eqref{eq:SC in g qiso} for $\tau = 1$. 

\textbf{Acknowledgements:} The authors are grateful to Shaoyun Bai, Daniel Pomerleano, Paul Seidel, Umut Varolgunes, and Iam Zemke for encouragement and helpful conversations and comments. They also thank the referee for helpful comments. M.E.A. and N.S. were supported by an ERC Starting Grant (award number 850713 -- HMS). N.S. was also supported by a Royal Society University Research Fellowship, the Leverhulme Prize, and a Simons Investigator award (award number 929034).

\section{Geometric Background}

In this section we introduce our conventions concerning quantum cohomology and Floer cohomology (we refer to \cite{mcduffsalamon,salamon} for additional background on these areas), and introduce the Hamiltonians which will play a key role in our main argument.

\subsection{Quantum Cohomology}

Recall that we assume $(M,\omega)$ to be a compact symplectic manifold which is positively monotone, so that 
\begin{equation*}
[\omega] = 2\kappa c_1(M) \quad \text{in}\ H^2(M,\mathbb{R}), \quad \kappa>0.
\end{equation*}
We define the graded Novikov ring $\Lambda^{\gr}_{\omega}$ to be the group ring $\bk[\pi_2(M)]$, i.e., the algebra generated by formal symbols $\{ e^A \ | \ A\in \pi_2(M)\}$, subject to the relation $e^A\cdot e^B = e^{A+B}$. We equip it with the integer grading $\abs{e^A} = 2c_1(A)$. 
We observe that it is \emph{degreewise} complete with respect to the action functional
\begin{equation}
    A_{\omega}(e^A) = \omega(A)
\end{equation}
(because the action is constant, by monotonicity). 
We note that $\Lambda^{\gr}_{\omega}$ is similar to the Novikov ring in \cite[\S 5]{Tian95}. 
It is however different from the more common Novikov $\Lambda_{\omega}$ obtained by taking the full completion, as decribed in \cite[\S 2]{PSS}; the latter is not a graded ring.

Let $H:S^1\times M\rightarrow \mathbb{R}$ be a time-dependent Hamiltonian. Its corresponding vector field $X_{H}$ is defined by the equation
\begin{equation}
    \iota_{X_H} \omega = -dH.
\end{equation}
We note that our convention here is identical to \cite{PL-theory}, but different from other references we use such as \cite{PSS} or \cite{BSV}.

The Hamiltonian $H$ is said to be non-degenerate if the time $1$ flow $\phi_1^{X_H}:M\rightarrow M$ has a graph $\text{Graph}(\phi_1^{X_H})\subseteq M\times M$ which is transverse to the diagonal $\Delta_M \subseteq M\times M$. This property ensures that the set of $1$-periodic orbits is a transversely cut compact manifold of dimension $0$.

A \emph{capped orbit} is a pair $(\gamma,u)$ where $\gamma$ is a $1$-periodic orbit of $X_{H}$, and $u$ is an equivalence class of maps $u:\mathbb{D}\rightarrow M$ such that $u|_{\partial \mathbb{D}} = \gamma$, where two such maps $u_1$ and $u_2$ are equivalent if 
\begin{equation*}
     u_1\#-u_2  = 0
\end{equation*}
in $\pi_2(M)$. 
We denote the collection of all capped $1$-periodic orbits by $\mathcal{P}_H$. It is exactly the set of critical points of the action functional
\begin{equation}\label{eq: action functional}
    A_H(\gamma,u) = -\int_{\mathbb{D}} u^*\omega + \int_{\partial\mathbb{D}} \gamma^*H.
\end{equation}

Assuming that $H$ is non-degenerate, we can associate with each capped orbit $x = (\gamma,u)$ an orientation line $o_x$ and its $\bk$-normalization $\abs{o_x}_{\bk}$, see \cite{ES1}.
Floer's complex associated with $H$ is the \emph{graded} completion of $\bigoplus_{x\in \mathcal{P}_H} \abs{o_x}_{\bk}$ with respect to the action filtration $A_H$,
\begin{equation}\label{Floer complex}
    CF^*(M,H) = \bigoplus_{i\in \mathbb{Z}}\overline{\bigoplus_{\substack{x\in \mathcal{P}_H\\\abs{x}=i}} \abs{o_x}_{\bk}}^{A_H}.
\end{equation}
It is a $\mathbb{Z}$-graded module over the graded Novikov ring $\Lambda^{\gr}_{\omega}$, where the module structure is
\begin{equation}
    e^A\cdot(\gamma,u) = (\gamma,u\# -A).
\end{equation}
We note in particular that
\begin{equation}\label{eq:ind_act_change}
    \begin{cases}
       &\abs{e^A\cdot(\gamma,u)} = \abs{(\gamma,u)} + 2c_1(A)\\
       &A_H(e^A\cdot(\gamma,u)) = A_H(\gamma,u) + \omega(A).
    \end{cases}
\end{equation}
Floer's complex is also equipped with a differential which counts maps $u:\mathbb{R}\times S^1 \rightarrow M$ satisfying the pseudo-holomorphic curve equation
\begin{equation}\label{Floer trajectory}
    \begin{cases}
       & \partial_s u + J_t (\partial_t u - X_{H}(u)) = 0\\
       & E(u)\defeq \int \abs{\partial_s u}^2 < +\infty,
    \end{cases}
\end{equation}
where $J_t$ is a generic $\omega$-compatible almost complex structure.
For each \emph{distinct} pair $x_-,x_+ \in \mathcal{P}_H$, define the moduli space
\begin{equation}\label{Floer trajectories}
    \mapsms(x_-,x_+) = \left\{ u\  \text{satisfies}\ \eqref{Floer trajectory},\ u_- \# [u] = u_+ \ \text{and}\ \ \lim_{s\rightarrow \pm\infty} u(s,t) = \gamma_{\pm}\right\}/\mathbb{R},
\end{equation}
where the action of $\mathbb{R}$ is by translation in the $s$-direction. Floer's differential on \eqref{Floer complex} is
\begin{equation}\label{eq: Floer differential}
    dx_+ = \sum_{x_-\neq x_+}\# \mapsms(x_-,x_+)\cdot x_-.
\end{equation}
Here, $\# \mapsms(x_-,x_+)$ is notation for the signed count of isolated points in this moduli space.
Moreover, for any $\Lambda^{\gr}_{\omega}$-module $R$, we define a Floer complex with coefficients in $R$,
\begin{equation}
    CF^*(M,H;R) = CF^*(M,H)\otimes_{\Lambda^{\gr}_{\omega}} R.
\end{equation}
Its cohomology $HF^*(M,H;R)$ is independent of $H$ and $J$, and we have a PSS-isomorphism (see \cite{PSS})
\begin{equation}
   HF^*(M,H;R)\cong H^*(M,R). 
\end{equation}
Under this isomorphism, the quantum product on $H^*(M,\Lambda^{\gr}_{\omega})$ corresponds to the pair-of-pants product on $HF^*(M,H)$. Throughout this work, we refer to Floer cohomology of a \emph{compact positive} symplectic manifold $(M,\omega)$ as quantum cohomology.

\subsection{Simple crossings divisors}\label{sec:scd}

Recall that a symplectic divisor $D\subseteq M$ is a (real) codimension 2 compact submanifold such that $\omega\vert_{D}$ is non-degenerate. A collection $(D_i)_{i=1}^N$ of symplectic divisors is said to be transverse if, for any subset $I\subseteq \{1,\dots,N\}$, the intersection
\begin{equation}
    D_I =\bigcap_{i\in I} D_i
\end{equation}
is a transversely cut smooth submanifold. In particular, the normal bundle of $D_I$ is given by 
\begin{equation}\label{normal bundle iso}
    N_M D_I = \bigoplus_{i\in I } N_MD_i\vert_{D_I}.
\end{equation}

For each $I$, the normal bundle $N_M D_I$ acquires a symplectic orientation using the normal bundle sequence
\begin{equation}
   0\rightarrow TD_I \rightarrow TM\vert_{D_I} \rightarrow N_M D_I \rightarrow 0,
\end{equation} 
where $TD_I$ and $TM\vert_{D_I}$ are oriented using the non-degenerate $2$-form $\omega$. However, this collection of orientations need not be coherent. Following \cite[\S 2]{TMZ}, the union $D = (D_i)_{i=1}^N$ is said to be 
\begin{itemize}
    \item[(i)] a simple crossings divisor if the isomorphisms \eqref{normal bundle iso} respect the symplectic orientation for all $I\subseteq \{1,\dots,N\}$.
    \item[(ii)] orthogonal if for every $i\neq j$, and every $x\in D_i\cap D_j$, the $\omega$-normal bundle $(T_x D_i)^{\perp \omega} \subseteq TM$ is a vector subspace of $T_x D_j$.
\end{itemize}

Let $D= (D_i)_{i=1}^N$ be an orthogonal simple crossings divisor such that
\begin{equation}
    2c_1(M) = \lambda_1 \pd(D_1) + \dotsi + \lambda_N \pd(D_N),
\end{equation}
where $\pd: H_{2n-2}(M,\mathbb{Q})\rightarrow H^2(M,\mathbb{Q})$ is Poincar\'e duality and $\lambda_i \in \mathbb{Q}_{>0}$ are positive rational numbers. This condition essentially says that $D$ supports (a multiple of) an effective anti-canonical divisor. 

Consider the exact symplectic manifold $X = M\backslash D$. By Lefschetz duality, the homology group $H_2(M,X,\mathbb{Z})$ is freely generated by small discs $(u_i)_{i=1}^N$ where each $u_i:\disk \rightarrow (M,X)$ intersects $D_i$ once transversely at $0\in\mathbb{D}$, and $u_i$ is disjoint from all $(D_j)_{j\neq i}$. Using the universal coefficient theorem, we have an isomorphism
\begin{equation}
    H^2(M,X,\mathbb{R}) \cong H_2(M,X,\mathbb{R})^{\vee}.
\end{equation}
Let $\pd^{\rel}(D_i)\in H^2(M,X,\mathbb{R})$ be the dual basis to $(u_i)_{i=1}^N$, and let
\begin{equation}
    [\omega]^{\rel}= \kappa_1 \pd^{\rel}(D_1) + \dotsi + \kappa_n \pd^{\rel}(D_N),\quad\text{where}\ \kappa_i = \kappa\cdot\lambda_i.
\end{equation}
The class $[\omega]^{\rel}$ is a lift of $[\omega]$, and similarly the classes $\pd^{\rel}(D_i)$ are lifts of $\pd(D_i)$ under the map
\begin{equation}
    H^2(M,X,\mathbb{R}) \rightarrow H^2(M,\mathbb{R}).
\end{equation}

Recall that we also have a de Rham chain model for relative cohomology groups given by
\begin{equation}
    \Omega(M,X,\mathbb{R}) = \text{Cone}(\Omega(M,\mathbb{R}) \rightarrow \Omega(X,\mathbb{R})) [-1].
\end{equation}
In particular, $\omega^{\rel}$ may be thought of as a pair $(\omega,-\theta')$ where $\theta'\in \Omega^1(X)$ such that $d\theta' = \omega$. By definition of this relative class, we have
\begin{equation}\label{wrapping number}
    \kappa_i = \int_{u_i} \omega - \int_{\partial u_i} \theta'.
\end{equation}
This is called the wrapping number of $\omega^{\rel}$ relative to the divisor $D_i$ and it is invariant under exact perturbations of $\theta'$. In the terminology of McLean (see \cite[Lemma 5.17, Appendix A]{McLean-growth-rate}), $(X,\theta')$ is a convex symplectic manifold because all $\kappa_i$ are positive.

\subsection{The Liouville structure}

\subsubsection{When $D$ is smooth}

Let $D\subseteq M $ be a \emph{smooth} symplectic divisor. By a Moser type argument (see \cite[Lemma 5.9]{McLean-growth-rate}), there is a neighbourhood $UD\subseteq M$ of $D$ which comes with a symplectic fibration map $\pi: UD \rightarrow D$, such that the fibers are symplectomorphic to a standard disc $D_{R} = \{r\leq R\}\subseteq\mathbb{C}$ of \emph{area} $R$. In our notation, $\mathbb{C}$ is given the symplectic form 
\begin{equation}
    \omega_{\mathbb{C}} = \rho d\rho\wedge d\theta \defeq dr\wedge d\phi,
\end{equation}
where $(\rho,\theta)$ are the usual polar coordinates, whereas $r=\pi\rho^2$ and $\phi =\theta/2\pi$. 
The radial coordinate $r:UD\rightarrow [0,R)$ generates a Hamiltonian vector field $X_r = \partial_{\phi}$ which satisfies the following properties:
\begin{itemize}\label{radial functions}
    \item[(i)] The field $X_r$ integrates to an $\mathbb{R}/\mathbb{Z}$-action on $UD$.
    \item[(ii)] The Hamiltonian action fixes all points of $D$.
    \item[(iii)] The Hamiltonian action is free on $UD\backslash D$.
\end{itemize}

With this setup, let $\theta'$ be a primitive for $\omega$ in $M\backslash D$, and let $F$ be a fiber of $\pi:UD\rightarrow D$. Since $F\backslash\{0\} \cong \mathbb{C}^*$, we must have a constant $\kappa$, and a function $f:F\backslash\{0\}\rightarrow \mathbb{R}$ such that
\begin{equation}
\theta'\vert_{F\backslash\{0\}} = rd\phi - \kappa d\phi + df.
\end{equation}
This constant $\kappa$ is the same as the wrapping number from \eqref{wrapping number}.
Therefore, if $Z$ is the Liouville field of $\theta = \theta'-df$, we have
\begin{equation}
    Z(r) = \theta(\partial_\phi) = r-\kappa.
\end{equation}
It follows that the Liouville field $Z$ points outwards along the boundary of $M\backslash UD$, provided that $\kappa>0$ and $R$ is sufficiently small.

\subsubsection{When $D$ is orthogonal simple crossings}

We recall the construction of the Liouville structure from \cite[\S 5]{McLean-growth-rate} and \cite[\S 2.3]{BSV}. Assume that we have a lift $\omega^{\rel} = (\omega,-\theta')$ of $\omega$ such that

\begin{equation}
    [\omega^{\rel}] = \sum_{i=1}^N \kappa_i \pd^{\rel}(D_i)\quad \text{in}\ H^2(M,X,\mathbb{R}),
\end{equation}
where all $\kappa_i$ are positive.

By \cite[Lemma 5.14]{McLean-growth-rate}, it is possible to construct open neighbourhoods $UD_i$ of $D_i$ with radial functions $r_i:UD_i\rightarrow [0,R)$ such that for all $i,j$,
\begin{itemize}
    \item[(i)] The intersection $UD_i\cap UD_j$ is invariant under the $\mathbb{R}/\mathbb{Z}$-action of $X_{r_i}$.
    \item[(ii)] Along the intersection $UD_i\cap UD_j$, the Hamiltonians $r_i$ and $r_j$ Poisson commute.
\end{itemize}

The assumptions on $D$ ensure that we can construct compatible neighbourhoods $UD_I$ of $D_I$ for each subset $I\subseteq \{1,\dots, N\}$. The fiber of the projection $\pi_I:UD_I \rightarrow D_I$ is a polydisc $F_I = \prod_{i\in I} D_R$ which carries a torus action by the commuting vector fields $(X_{r_i})_{i\in I}$. McLean's lemma constructs an exact perturbation $\theta$ of $\theta'$ such that on each punctured fiber $F_I^* = \prod_{i\in I} (D_R\backslash \{0\})$,
\begin{equation}
    \theta\vert_{F_I^*} = \sum_{i\in I} (r_i-\kappa_i)d\phi_i.
\end{equation}

In particular, if $Z$ is the Liouville vector field of $\theta$, we have that 
\begin{equation}
    Z(r_i) = r_i - \kappa_i\quad \text{on } UD_i, \text{ for  all } i = 1,2,\dots,N.
\end{equation}
 Since all $\kappa_i$ are positive, there is a positive parameter $R_0\in [0,\min_i \kappa_i)$ such that $Z$ points outwards along the boundary of $M\backslash \cup_{i=1}^N \{r_i \le R\}$ for all $R \in [0,R_0)$.

Let $\mathbb{L}$ be the $Z$-skeleton of $X$, i.e. the collection of points $p\in X$ such that the flow $\phi^Z_t (p)$ is defined for all time $t$. Any point $p\in X\backslash\mathbb{L}$ has a Liouville flow that eventually enters one of the open sets $UD_i$. Define $UD_i^{\max}$  to be the open set consisting of $D_i$ together with all the points in $X$ whose Liouville flow enters $UD_i$. Note that $\cup_i UD_i^{\max} = M\backslash \mathbb{L}$, and that each $r_i:UD_i \rightarrow [0,R)$ extends smoothly to $r_i^{\max}: UD_i^{\max} \rightarrow [0,\kappa_i)$ using the Liouville flow. One checks that the Hamiltonians $r_i^{\max}$ still define commuting $\mathbb{R}/\mathbb{Z}$-actions on $UD_i^{\max}$.

\subsection{The radial coordinate} \label{subsec: radial coord}

Next, we briefly recall a construction from \cite[\S 4]{BSV} of a \emph{smooth} radial function $\rho: X\rightarrow \mathbb{R}$ such that
\begin{equation}
Z(\rho) = \rho.
\end{equation}
A natural candidate $\rho^0$ is given by $e^{-t}$, where $t(x)$ is the blow-up time of the flow $\phi^{Z}_t(x)$. This is well-defined for any $x \in X\backslash \mathbb{L}$ where $\mathbb{L}$ is the $Z$-skeleton of $X$. For instance if $x \in UD_i$, then $\rho^0(x) \leq 1 - r_i(x)/\kappa_i$. In fact, it is not difficult to see that
\begin{equation}
    \rho^0(x) = \max\left\{ 1-\frac{r_i(x)}{\kappa_i}\ \bigg|\  1\leq i\leq N \ \text{such that}\ x\in UD_i^{\max}\right\}.
\end{equation}
We also extend $\rho^0$ continuously to all of $M\backslash \mathbb{L}$ by setting $\rho^0\vert_{D} = 1$. 

In \cite[\S 4]{BSV}, the authors construct a regularization $\rho^R$ of $\rho^0$ depending on a smoothing parameter $R\in (0,R_0)$ such that each $\rho^R:M\backslash\mathbb{L}\rightarrow \mathbb{R}$ is a smooth function satisfying $Z(\rho^R) = \rho^R $ on $X\backslash\mathbb{L}$, and $\rho|_{D}\geq 1$. The construction of this smoothing is rather technical, so we only recall its most important properties and we refer the reader to \cite{BSV} for proofs.

\begin{lemma}\cite[Lemma 4.14]{BSV}
    There exist smooth functions $\nu_i: M\backslash\mathbb{L} \rightarrow \mathbb{R}_{\geq 0}$ for $i\in\{1,\dots,N\}$, such that $\nu_i$ is supported in $UD_i^{\max}$ and for all $p\in M$
    \begin{equation}
        d_p\rho = -\sum_{i=1}^N \nu_i(p) d_p r_i^{\max}. 
    \end{equation}
    Moreover, $d\nu_i(X_{r_j}) = 0 $ for all $i,j \in \{1,\dots, N\}$.
\end{lemma}

The last property in the Lemma above ensures that $\nu_i$ is constant along the Hamiltonian orbits of $\rho$. More generally, let $h:\mathbb{R}\rightarrow \mathbb{R}$ be a smooth function which is constant near $0$. Then $h(\rho):M\rightarrow \mathbb{R}$ is a globally defined smooth Hamiltonian, and we denote its flow by $\Phi^{h(\rho)}_t$. A short calculation shows that $\nu^h_i \defeq h'(\rho)\nu_i$ is constant along the flow $\Phi^{h(\rho)}_t$. In fact, we can calculate the whole flow in terms of $\nu^h \defeq (\nu_1^h,\dots,\nu_N^h): M\rightarrow \mathbb{R}^N$. For each $p\in M$, define 
\begin{equation}
    I_p = \{ i\ : \ \nu_i(p) \neq 0 \}
\end{equation}
so that $p\in UD^{\max}_{I_p}$. Then, the time-$1$ flow of $p$ is given by
\begin{equation}
        \Phi^{h(\rho)}_1(p) = -\nu^h(p) \cdot p, 
\end{equation}
where $\nu^h(p)$ is projected to $(\mathbb{R}/\mathbb{Z})^{I_p}$ and it acts on $p \in UD^{\max}_{I_p}$ using the torus action.

\begin{corollary}\cite[Cor 4.15]{BSV}
    The $1$-periodic orbits of $h\circ\rho$ correspond to points $p\in M$ such that for all $i=1,\dots,N$,
    \begin{equation}
        \nu_i^h(p)\in \mathbb{Z}\quad\text{or}\quad p\in D_i.
    \end{equation}
\end{corollary}

\subsection{Action-Index computations}

Let $\gamma$ be a $1$-periodic orbit of $h\circ\rho$. Notice that the function $\nu^h$ is constant along $\gamma$. Moreover, if $h'\geq 0$, then the components $\nu^h_i$ are non-negative and may be thought of as the winding numbers of $\gamma$ around each $D_i$. We set,
\begin{equation}
I_{\gamma} =  \{ i \ : \ \nu^h_i(\gamma) \neq 0 \}.
\end{equation}

The orbit $\gamma$ has an associated outer cap $u_{\text{out}}$ which is unique up to homotopy relative to the boundary. The construction of this cap is detailed in \cite[Lemma 2.14]{BSV}. If $\gamma$ is constant, $u_{\text{out}}$ is just the constant cap. Otherwise, the cap $u_{\text{out}}$ is obtained by flowing $\gamma$ using the Liouville vector field towards $D_{I_{\gamma}}$ and then capping the resulting cylinder in a small contractible chart. The action and index of such orbits is computed in \cite[\S 4.8]{BSV}.

\begin{lemma}
    With respect to the outer cap, we have that
    \begin{align}
        A(\gamma,u_{out}) &= h(\rho(\gamma)) + \sum_{i=1}^N \nu_i^h(\gamma)r_i^{\max}(\gamma),\\
        CZ(\gamma,u_{\text{out}}) &= 2\sum_{i=1}^N \lceil \nu_i^h(\gamma)\rceil + \frac{1}{2}\text{sign}(\text{Hess}_{\gamma}),
    \end{align}
    where $\text{Hess}_{\gamma}$ is a particular symmetric matrix associated with $\gamma$.
\end{lemma}

Choose $0<\eta \ll 1$, and set $\sigma^+ = 1-\eta$, $\sigma = 1-2\eta$, $\sigma^- = 1-3\eta$. 
Because the Liouville field satisfies $Z(\rho^R) = \rho^R$, the compact subset $X_\sigma = \{\rho \defeq\rho^R \leq \sigma \}$ is a Liouville domain. The boundary $\partial X_\sigma$ inherits a contact structure from the Liouville $1$-form $\theta$, and we choose a positive slope $\lambda$ such that
\begin{equation}\label{eq: slope lambda}
   \lambda \mathbb{Q}\cap \spec(\partial X_{\sigma}) = \emptyset.
\end{equation}
Here, the action spectrum $\spec(\partial X_{\sigma}) \subseteq \mathbb{R}_{>0}$ is the subset of all lengths of closed Reeb orbits. Such $\lambda > 0$ exists because $\spec(\partial X_{\sigma})$ is a measure zero subset of the reals, see \cite[Lemma 2.2]{Oh-measurezero} for instance. 
Next, let $h$ be a non-decreasing smooth function such that (see Figure \ref{fig: Base Hamiltonian})
\begin{equation}
    \begin{cases}
        h(\rho) = 0 &\text{if}\ \rho \leq 1 - 5\eta, \\
        h(\rho) = \lambda(\rho - (1-4\eta)) &\text{if}\ \rho \geq \sigma_-.
    \end{cases}
\end{equation}

\begin{figure}
    \centering
    \begin{tikzpicture}[x=0.75pt,y=0.75pt,yscale=-1,xscale=1]
        %uncomment if require: \path (0,488); %set diagram left start at 0, and has height of 488
        
        %Shape: Axis 2D [id:dp649979734851386] 
        \draw  (143,331.8) -- (560,331.8)(184.7,87) -- (184.7,359) (553,326.8) -- (560,331.8) -- (553,336.8) (179.7,94) -- (184.7,87) -- (189.7,94)  ;
        \draw   (335,327) -- (335,337) ;
        \draw   (415,327) -- (415,337) ;
        \draw   (495,327) -- (495,337) ;
        %Straight Lines [id:da69642450898848] 
        % \draw    (184.7,332) -- (363,332) ;
        %Curve Lines [id:da43426059602264333] 
        \draw    (335,332) .. controls (363,329) and (388,315) .. (400,296) ;
        %Straight Lines [id:da18004581074640402] 
        \draw    (400,296) -- (531,103) ;
        \draw   (455,327) -- (455,337) ;
        \draw   (535,327) -- (535,337) ;
        
        \draw   (375,327) -- (375,337) ;
        
        % Text Node
        % \draw (287,335) node [anchor=north west][inner sep=0.75pt]   [align=left] {$\sigma_-/2$};
        % Text Node
        \draw (400,335) node [anchor=north west][inner sep=0.75pt]   [align=left] {$\sigma_-$};
        % Text Node
        \draw (480,335) node [anchor=north west][inner sep=0.75pt]   [align=left] {$\sigma_+$};
        % Text Node
        \draw (520,335) node [anchor=north west][inner sep=0.75pt]   [align=left] {$1$};
   
        \draw (566,325) node [anchor=north west][inner sep=0.75pt]   [align=left] {$\rho$};
        % Text Node
        \draw (172,64) node [anchor=north west][inner sep=0.75pt]   [align=left] {$h(\rho)$};
        % Text Node
        \draw (440,335) node [anchor=north west][inner sep=0.75pt]   [align=left] {$\sigma$};

    \end{tikzpicture}
    \caption{Linear Hamiltonian, with the points $1$, $\sigma_+ = 1-\eta$, $\sigma = 1-2\eta$, $\sigma_- = 1-3\eta$, $1-4\eta$, and $1-5\eta$ marked on the $\rho$-axis.}
    \label{fig: Base Hamiltonian}
\end{figure}

For the purposes of Floer theory, we need to perturb the Hamiltonian $h\circ \rho$ so that the $1$-periodic orbits are non-degenerate. An example of such perturbations are constructed in \cite[Lemma 4.17]{BSV}. We require that our perturbations are within the class $C^{\infty}(M,D)$ of Hamiltonians preserving $D$, i.e.
\begin{equation}
    C^{\infty}(M,D) = \{ H:M\rightarrow \mathbb{R}\ | \ \forall p\in D,\quad X_{H,p} \in T_pD\}.
\end{equation}
If $D = (D_i)_{i=1}^N$ is a simple crossings divisor, it should be understood that
\begin{equation}
    C^{\infty}(M,D) = \bigcap_{i=1}^N C^{\infty}(M,D_i).
\end{equation} 
\begin{lemma}\label{lemma: Hamiltonians}
    Let $(\epsilon_n)_n$ be a sequence of positive real numbers. There is a sequence $\bH = (H_n)_{n\geq 1}$ of non-degenerate Hamiltonians $H_{n,t} = nh(\rho) + K_{n,t} \in C^{\infty}(M,D)$ such that $K_{n,t}\leq 0$ and
    \begin{itemize}
        \item[(1)] The functions $K_{n,t}$ vanish on the neck $\{\sigma_- \leq \rho \leq \sigma_+\}$.
        \item[(2)] There is a sequence  $C_n:M\rightarrow \mathbb{R}$ such that $K_{n-1,t} \leq C_n \leq K_{n,t}$.
        \item[(3)] The orbits of the $H_{n,t}$ contained in $X_\sigma$ are pairwise disjoint.    
        \item[(4)] Every orbit $(\gamma,u)$ of $H_{n,t}$ has a nearby orbit $(\overline{\gamma},\overline{u})$ of $nh(\rho)$ s.t.
        \begin{itemize}
            \item[(i)] $\abs{A(\gamma,u)-A(\overline{\gamma},\overline{u})} < \epsilon_{n}$.
            \item[(ii)] $\abs{\CZ(\gamma,u)-\CZ(\overline{\gamma},\overline{u})} \leq \dim(\ker(d\Phi^1_{h\circ\rho}-\id))/2$.
        \end{itemize}
    \end{itemize}
\end{lemma}

\begin{proof}
    The Hamiltonians $H_{n,t}$ are constructed as in \cite[Lemma 4.17]{BSV}; the only new part which is potentially not routine is that the Hamiltonians may be chosen to lie in $C^\infty(M,D)$. This can be done using Lemma \ref{lem:pert_along_D} in place of the usual argument that nondegeneracy can be achieved by an arbitrarily small perturbation in $C^\infty(M)$. 
\end{proof}

\begin{lemma}\label{lem:pert_along_D}
    Let $H_t \in C^\infty(M,D)$ be a family of Hamiltonians parametrized by $t \in S^1$, and let $U \subset M$ be an open set such that all orbits of $H_t$ which intersect $U$ are contained in it. 
    Then there exists $\delta_t \in C^\infty(M,D)$, supported in $U$, such that all orbits of $H_t+\delta_t$ contained in $U$ are nondegenerate. 
    Furthermore, $\delta_t$ may be chosen to be arbitrarily small in any $C^\ell$ norm.
\end{lemma}
\begin{proof}
    We will assume that $U = M$ for simplicity; the proof is the same for general $U$.
    
    Because the Hamiltonian flow of $H_t$ preserves each divisor $D_i$ by the definition of $C^\infty(M,D)$, for each orbit $\gamma$ there exists $I_\gamma$ such that $\gamma$ is contained in $D_{I_\gamma}$, and does not intersect $D_j$ for $j \notin I_\gamma$. 
    Now let us assume that all orbits $\gamma$ of $H_t+\delta_t$ with $\abs{I_\gamma} \geq k+1$ are nondegenerate. 
    We claim that for each $I$ with $\abs{I}=k$, we may choose a perturbation of $\delta_t$ which is supported away from a neighbourhood of $D_j$ for all $j \notin I$, which makes the orbits $\gamma$ with $I_\gamma = I$ nondegenerate. 
    Note that this perturbation does not affect orbits on $D_J$ for $\abs{J} \ge k$ and $J \neq I$. 
    Therefore, by applying such a perturbation to each $I$ with $\abs{I} =k$, we may arrange that all orbits $\gamma$ with $\abs{I_\gamma} \ge k$ are nondegenerate. 
    The result then follows by (decreasing) induction on $k$.
    
    It remains to establish the claim. 
    The set of orbits $\gamma$ with $I_\gamma = I$ is closed, and contained in $D_I \setminus \cup_{j \notin I} D_j$; therefore it is contained in some open set $U_I$ which is disjoint from a neighbourhood of $\cup_{j \notin I} D_j$. 
    By the standard argument, we may choose a perturbation of $\delta_t|_{D_I}$, supported in $D_I \cap U_I$, so that all orbits contained in $D_I \cap U_I$ are nondegenerate when considered as an orbit inside $D_I$. 
    We then extend this perturbation to $M$ by pulling it back to the neighbourhood $UD_I$ of $D_I$ by the fibration $\pi_I$, whose fibres we recall are symplectically orthogonal to $D_I$; then cutting off with a bump function depending on the variables $(r_i)_{i \in I}$. 
    This ensures that the extension is contained in  $C^\infty(M,D)$, and by choosing the support of the bump function sufficiently small, we may ensure that the extension is supported in $U_I$. 
    The orbits $\gamma$ of this extension with $I_\gamma = I$ are precisely the orbits of the restriction to $U_I$; and these are nondegenerate as orbits inside $D_I$. 
    That is, if we write 
    $$ D \phi^1_{H_t+\delta_t} - \id= \left[ \begin{matrix}
                                    A - \id& B \\
                                    C & D - \id
                                \end{matrix} \right]\in \mathrm{End}\left( TD_I \oplus ND_I \right),       $$
    then $A-\id$ is nonsingular. 
    Note that the return map preserves $TD_I$, so $C=0$. 
    Thus, in order to verify that $D\phi^1_{H_t+\delta_t} - \id$ is nonsingular, it suffices to arrange that $D - \id$ is nonsingular. 
    We claim that this can be achieved by adding $\epsilon \cdot \eta(t)\cdot(\sum_{i \in I} r_i^2)$ to $\delta_t$, where $\eta(t)$ is an approximation to a delta function supported in $(1 - \epsilon',1)$. Let $D_\epsilon$ denote the resulting component of the return map. By choosing $\eta(t)$ sufficiently close to a delta function, we can make $D_\epsilon$ arbitrarily close to the composition of $R_\epsilon \cdot D_0$, where $R_\epsilon$ is the `rotation by $\epsilon$' map, in any $C^\ell$ norm. As the function $\det(R_\epsilon \cdot D_0 - \id)$ is a polynomial in $\cos(\epsilon)$ and $\sin(\epsilon)$ which is non-constant (unless $D_0 = 0$, in which case no perturbation is necessary to achieve nondegeneracy), it has non-vanishing jet at the origin $\epsilon=0$. 
    Therefore, by choosing $\eta$ to approximate a delta function sufficiently closely, we may arrange that $\det(D_\epsilon - \id)$ also has non-vanishing jet at the origin, and therefore is not identically $0$ in any neighbourhood of the origin. 
    Therefore we may achieve nondegeneracy for arbitrarily small values of $\epsilon$; in particular, by taking $\epsilon$ sufficiently small, we may achieve nondegeneracy using a perturbation of $\delta_t$ which is arbitrarily small in the $C^\infty$ topology.
    This completes the proof of the claim.
\end{proof}

Because of our choice of $\lambda$, the Hamiltonian $nh(\rho)$ does not have $1$-periodic orbits in the region $\{ \sigma_-\leq \rho \leq \sigma_+\}$. In particular, we have a dichotomy of $1$-periodic orbits of the Hamiltonians $H_n: M\rightarrow \mathbb{R}$:
\begin{itemize}
    \item[-] $SH$-orbits, which are the orbits contained in $X_{\sigma}$.
    \item[-] $D$-orbits, which are the orbits living outside of $X_{\sigma}$.
\end{itemize}

When the parameter $\eta$ (which appeared in the definition of $h(\rho)$) is chosen sufficiently small, the $D$-orbits tend to have a higher index relative to their action. The following is one of the key results of \cite{BSV}.

\begin{lemma}\cite[Lemma 5.5]{BSV}\label{lemma: D-orbits are P-negative}
    If $\eta$ is sufficiently small, there is a smoothing parameter $R\in (0,R_0)$ which is close to $0$, and a constant $C \defeq C(\lambda,\sigma,R)>2$ such that all capped $D$-orbits $x = (\gamma,u)$ satisfy the inequality
    \begin{equation}
        C\cdot n_x + \kappa^{-1} A(x) - \abs{x} < 0,
    \end{equation}
    where $\kappa$ is the monotonicity constant from \eqref{eq:monotonicity}.
\end{lemma}

From the proof of \cite[Lemma 5.5]{BSV}, the constants $C(\lambda,\sigma,R)$ are multiplicative in $\lambda$ so that
\begin{equation}
    C(m\lambda,\sigma,R) = mC(\lambda,\sigma,R).
\end{equation}
Therefore, by choosing a sufficiently large slope $\lambda$ (see \eqref{eq: slope lambda}), we can arrange for
\begin{equation}\label{eq: C>2, version 1}
    C + \kappa^{-1}\min_y C_n(y) > 2 \quad \text{for all} \ n\in \mathbb{N}.
\end{equation}

\section{Pseudo-holomorphic curve theory}

In this section we introduce the moduli spaces of pseudoholomorphic curves which we use to construct the Maurer--Cartan element.

\subsection{Rational curves with aligned framings}

Let $\domms_{0,1+d}$ be the Deligne--Mumford space of smooth genus $0$ curves with $1+d$ marked points, and $\overline{\domms}_{0,1+d}$ its compactification. A 
\emph{framing} on a curve $(C,z_0,\dots,z_d) \in \domms_{0,1+d}$ is the choice of a tangent direction $\theta_i\in \mathbb{R}\mathbb{P}(T_{z_i}C)$ for each one of its marked points. The framing is said to be \emph{aligned} if for each $i = 1,\dots, d$, there is an isomorphism $\phi: C\rightarrow \mathbb{P}^1$ such that $\phi(z_0) = \infty$, $\phi(z_i) = 0$, and both $\phi_*(\theta_0)$ and $\phi_*(\theta_i)$ point along the positive real direction. The moduli space of smooth genus $0$ curves with $1+d$ marked points and aligned framings is
\begin{equation}
    \domms^{\al}_d = \text{Conf}_d (\mathbb{C})/\text{Aff}(\mathbb{C},\mathbb{R}_{>0}),
\end{equation}
where $\text{Aff}(\mathbb{C},\mathbb{R}_{>0}) = \{ z\mapsto az + b \ | \ a\in \mathbb{R}_{>0} \ \text{and} \ b\in \mathbb{C} \}$.

The moduli space $\domms^{\al}_d$ admits a Deligne--Mumford type compactification $\overline{\domms}^{\al}_d$ which is a smooth manifold with corners. It is an $S^1$-bundle over the real blowup of the Deligne--Mumford space $\overline{\domms}_{0,1+d}$ over the normal crossings divisor $\overline{\domms}_{0,1+d}\backslash \domms_{0,1+d}$, see \cite[\S 2]{ES1}.

Given a map $\bp : F \rightarrow \{1,\dots,d\}$, we can also construct a moduli space $\domms^{\al}_{d,\bp}$ of smooth genus $0$ curves with aligned framings and $\bp$-flavors. The $\bp$-flavor datum is a map $\psi: F\rightarrow \text{Isom}(C,\mathbb{P}^1)$ such that for each $f\in F$,
\begin{equation}
   \psi(f)(z_0) = \infty,\quad \psi(f)(z_{\bp(f)}) = 0,
\end{equation}
and $\psi(f)_*(\theta_0)$ points along the positive real direction.  The parametrization $\psi(f)$ is called a sprinkle for the marked point $z_{\bp(f)}$. 

A set of weights for a curve $C \in \domms^{\al}_{d,\bp}$ is a tuple of $1+d$ \emph{non-negative} integers $\bn = (n_0,n_1,\dots,n_d)$ such that
\begin{equation}
    n_0 = n_1+\dots+n_d + \abs{F}.
\end{equation}
The moduli space of $\bp$-flavored smooth genus $0$ curves with aligned framings in weight $\bn$ is simply $\domms^{\al}_{d,\bp,\bn} \defeq \domms^{\al}_{d,\bp}$. It admits a compactification $\overline{\domms}^{\al}_{d,\bp,\bn}$ modelled on semi-stable $d$-leafed trees described in \cite{ES1}, and this compactification is a smooth manifold with corners.

We will call a marked point \emph{forgettable} if it has weight $0$ and carries no sprinkle; these points play a special role in our constructions, analogous to \cite[Section 7]{perutz2022constructing} to which we refer for a more detailed account of an analogous setup. 
By forgetting all forgettable marked points, the triple $(d,\bp,\bn)$ reduces to a triple $(d^{\red},\bp^{\red},\bn^{\red})$. 
Moreover, we have a forgetful map
\begin{equation} \label{forgetful map}
    \overline{\domms}^{\al}_{d,\bp,\bn} \xrightarrow[]{\text{forgetful}}\overline{\domms}^{\al}_{d^{\red},\bp^{\red},\bn^{\red}}.
\end{equation}
This is well-defined unless $d^\red = 1$, $F= \emptyset$, $n^\red_1 > 0$; i.e., when forgetting the forgettable marked points gives a Floer cylinder. 
We must treat this special case separately throughout our constructions.

\begin{example}\label{framed cap}
    \emph{The framed cap} is the smooth genus $0$ curve $(C,z_0,z_1;\theta_0,\psi_1)$ where $z_0$ has weight $1$ and $z_1$ has weight $0$ and $z_1$ carries a sprinkle.
\end{example}

Recall that a positive/negative cylindrical end for a curve $C_r$ represented by $r \in \domms^{\al}_{d,\bp,\bn}$ at a point $z\in C_r$ is a holomorphic embedding
\begin{equation}
    \epsilon_{\pm}: Z^{\cl}_{\pm}\rightarrow C\backslash\{z\}
\end{equation}
of a positive/negative semi-infinite cylinder $Z^{\cl}_{\pm} = \{ (s,t)\in \mathbb{R}\times S^1 \ | \ \pm s \geq 0 \}$
such that $\lim_{\pm s \rightarrow \infty} \epsilon_{\pm}(s,t) = z$. A set of cylindrical ends for $r$ is the choice of a negative cylindrical end for the output $z_0$ if $n_0>0$, and positive cylindrical ends for positively weighted inputs $\{z_i\ | \ \text{if}\ i\geq 1 \text{ and } n_i\geq 1 \}$. We denote by $\Sigma_r$ the open Riemann surface obtained by removing all marked points with positive weights from $C_r$,
\begin{equation}
    \Sigma_r = C_r\backslash \{z_i \ |\ i\geq 1 \text{ and } n_i>0 \}.
\end{equation}

\subsection{Coherent perturbation data}

A perturbation datum for $r \in \domms^{\al}_{d,\bp,\bn}$ is a tuple $\mathscr{F}_r = (\mathcal{K}_r,J_r)$ consisting of a Hamiltonian perturbation and an almost complex structure, both of which are domain-dependent on the associated Riemann surface $\Sigma_r$. Suppose now that we have a nodal curve $r = (r_+,z_+)\#(r_-,z_-) \in \partial \overline{\domms}^{\al}_{d,\bp,\bn} $ with two components $r_{\pm} \in \domms^{\al}_{d_\pm,\bp_\pm,\bn_\pm}$ glued at the points $z_\pm$. Assume further that we have a perturbation datum $\mathscr{F}_{r_\pm}$ for each component $r_\pm$. We can construct a perturbation datum $\mathscr{F}_r$ for $r$ as follows:
\begin{itemize}
    \item[-] If $z_+$ has positive weight, then $\mathscr{F}_{r}$ agrees with $\mathscr{F}_{r_\pm}$ on the component $r_\pm$.
    \item[-] If $z_+$ has weight $0$, then $\mathscr{F}_r$ agrees with $\mathscr{F}_{r_+}$ on the component $r_+$, but it equals $(0,J_{r_-}\vert_{z_-})$ on the component $r_-$.
\end{itemize}
This construction easily generalizes to nodal curves $r\in \partial \overline{\domms}^{\al}_{d,\bp,\bn}$ which have multiple components (compare the \textbf{(Constant on spheres)} condition on perturbation data in \cite[Definition 5.4]{perutz2022constructing}). 

In this section, we construct perturbation data varying smoothly in $r$ satisfying the following conditions:
\begin{itemize}
    \item[] \textbf{(Consistency)} If $r_1,\dots,r_c$ are the irreducible components of $r$, then $\mathscr{F}_{r}$ is the perturbation datum induced by $(\mathscr{F}_{r_i})_{i=1}^c$ as described above.
    \item[] \textbf{(Forgetfulness)} The perturbation data $\mathscr{F}_r$ are constant along the fibers of the forgetful map \eqref{forgetful map}, whenever the latter is well-defined; in the case where forgetting all forgettable marked points gives a Floer cylinder, the perturbation data coincide with the Floer data, and in particular are independent of the position of the forgettable marked points.
    \item[] \textbf{(Equivariance)} The perturbation data $\mathscr{F}_r$ are $\Sym(d,\bp)$-invariant. 
\end{itemize}
We emphasize that the requirement of \textbf{(Consistency)} is slightly different from the corresponding requirement described in \cite[\S 9]{PL-theory} (necessarily so, as the latter would be incompatible with \textbf{(Forgetfulness)}). Indeed, we do not always require compatibility of perturbation data with gluing operations along cylindrical ends near the boundary $\partial\overline{\domms}^{\al}_{d,\bp,\bn}$. In fact, we do not even make use of cylindrical ends at marked points that have weight zero.
  
We do not specify perturbation data on the moduli spaces where all the weights $n_i$ are $0$; these moduli spaces do not appear in our construction (we deal with sphere bubbles in a different way: as we have already seen, when a sphere bubble develops along an edge with weight zero, it is equipped with an almost complex structure which is constant over the sphere, but this constant value may depend on where the sphere bubble was attached).

\subsection{Hamiltonian perturbations}

Recall that a subclosed $1$-form for $r\in \domms^{\al}_{d,\bp,\bn}$ is an element $\gamma_r \in \Omega^1(\Sigma_r)$ such that
\begin{equation}
    \begin{cases}
        & d\gamma_r \leq 0,\\
        & \epsilon_i^*\gamma_r = n_idt \quad\text{for}\ i\in Z^c(\bn)\text{ and } \pm s \rightarrow \infty,
    \end{cases}    
\end{equation}
where $Z^c(\bn) = \{i \in \{0,1,\dots,d\}\ |\ n_i \neq 0 \}$. Using the same approach of \cite{ES1},  we construct a coherent universal choice of cylindrical ends and subclosed $1$-forms on the moduli spaces $\domms^{\al}_{d,\bp,\bn}$ which are $\Sym(d,\bp)$-invariant. 
We require that the cylindrical ends and subclosed $1$-forms be compatible with forgetful maps \eqref{forgetful map} when they are well-defined, in the sense that the choices on $\domms^\al_{d,\bp,\bn}$ are obtained from those on $\domms^\al_{d^\red,\bp^\red,\bn^\red}$ by pulling back via the forgetful map. 
In the special case that forgetting all forgettable marked points gives a Floer cylinder, we require that the cylindrical ends extend to a parametrization of the whole curve by the cylinder $Z^\cl = \R \times S^1$ (cf. \cite[Definition 7.10]{perutz2022constructing}), and the $1$-forms are equal to $n_i dt$. Finally, we shrink the cylindrical ends so that they are disjoint from all marked points; 
Such $1$-forms and cylindrical ends may be constructed by the usual argument, together with the argument of \cite[Lemma 5.12]{perutz2022constructing} to arrange compatibility of the cylindrical ends with the forgetful maps. 

For $r\in \domms^{\al}_{d,\bp,\bn}$, an admissible Hamiltonian perturbation is a Hamiltonian valued $1$-form $\mathcal{K}_r \in \Omega^1(\Sigma_r,C^{\infty}(M,D))$ of the type
\begin{equation}\label{eq:KrFr}
    \mathcal{K}_r = (h(\rho) + F_r)\otimes \gamma_r,
\end{equation}
where $F_r$ is a domain-dependent Hamiltonian such that
\begin{equation}\label{eq:Kr_cyl}
    \begin{cases}
        F_r = 0 & \quad\text{on the neck region } \sigma_-\leq\rho\leq \sigma_+,\\
        \epsilon_i^*\mathcal{K}_r = H_{n_i}dt & \quad\text{for all } i \in Z^c(\bn) \text{ and } \pm s \gg 0.
    \end{cases}
\end{equation}

This type of Hamiltonian perturbation has an associated curvature term,
\begin{equation}
    R(\mathcal{K}_r) \defeq d_{\Sigma_r}(\mathcal{K}_{r}).
\end{equation}

In \cite[Corollary 4.3]{ES1}, we explain how to construct a consistent universal choice of perturbations $\mathcal{K}_r$ across all the moduli spaces $\domms^{\al}_{d,\bp,\bn}$ such that
\begin{equation}\label{curvature estimate}
    R(\mathcal{K}_r) \leq C_{\bn} d\gamma_r.
\end{equation}
In the curvature estimate above, $C_{\bn}$ is a constant given by
\begin{equation}
    C_{\bn} = 
    \begin{cases}
        0 \quad & \text{ if } \abs{Z^c(\bn)} = 2\\
        \min_i \widetilde{C}_{n_i}\quad & \text{ otherwise, }
    \end{cases}
\end{equation}
where
\begin{equation}
    \widetilde{C}_n =
    \begin{cases}
        C_0 \quad & \text{ if } n = 0\\
        \min_y C_n(y)/n \quad & \text{ if } n\geq 1.
    \end{cases}
\end{equation}
The sequence of autonomous Hamiltonians $(C_n)_{n \geq 1}$ is the same sequence from Lemma \ref{lemma: Hamiltonians}, and $C_0$ is a negative constant such that $C_0\leq C_1$. We choose $C_0$ sufficiently close to $C_1$ so that (compare with \eqref{eq: C>2, version 1})
\begin{equation}\label{eq: C>2}
    C + \kappa^{-1} C_{\bn} > 2 \quad \text{for all}\ \bn. 
\end{equation}

\subsection{Almost complex structures}\label{subsec: Almost complex structures}

Let $J$ be an almost complex structure on $M$ such that
\begin{itemize}
    \item[(1)] $J$ is $\omega$-tame and respects the snc divisor $D$, i.e. for all $i=1,\dots,N$ and $p\in D_i$, we have that $JT_pD_i = T_pD_i$
    \item[(2)] $J$ is of contact type on the neck region so that 
    \begin{equation}
        d\rho = \theta\circ J \quad \text{ when } \sigma_-\leq \rho \leq \sigma_+.
    \end{equation} 
\end{itemize}
Let $\mathscr{J}'$ be the contractible space of $\omega$-tame perturbations of $J$ defined in \cite[Section 3.3]{perutz2022constructing}, and $\mathscr{J} \subset \mathscr{J}'$ the subspace of perturbations which are equal to $J$ on the neck region $\{\sigma_- \leq \rho \leq \sigma_+\}$.  

\begin{remark}
Here are the key properties of $\mathscr{J}$:
\begin{enumerate}
    \item Any $J \in \mathscr{J}$ satisfies (1) (which allows us to use positivity of intersection) and (2) (which allows us to use the maximum principle).
    \item The space $\mathscr{J}$ is contractible, allowing for the inductive construction of perturbation data.
    \item The space $\mathscr{J}$ is sufficiently large to achieve all the transversality we need; in particular, no non-constant holomorphic curves $u$ can be contained entirely inside the neck region where $J$ is fixed, by exactness.
\end{enumerate}    
\end{remark}

For each natural number $n\in \mathbb{N}$, we choose a time-dependent $J_n \in \mathscr{J}$ such that the Floer complex $CF^*(X_{\sigma},H_n|_{X_{\sigma}},J_n)$ is well-defined, and furthermore such that there do not exist any Chern-number-1 $J_{n,t}$-holomorphic spheres passing through $\gamma(t)$, for any $SH$-orbit $\gamma$ of $H_n$ (this is true for generic $J_n$, cf. \cite[Theorem 3.1 (ii)]{HoferSalamon}). Due to the divisor restrictions on $J_n$, we do not claim that Floer's differential is well-defined on $CF^*(M,H_n,J_n)$.

Suppose now that $n_0>0$ and let $r\in \domms^{\al}_{d,\bp,\bn}$ be a moduli parameter. An admissible almost complex structure for $r$ is a domain-dependent $J_r \in \mathscr{J}$ such that
\begin{equation}\label{eq:Jr_cyl}
    \epsilon_{i}^* J_r \longrightarrow J_{n_i} \quad\text{ for all } i\in Z^c(\bn) \text{ and } \pm s \to \infty,
\end{equation}
asymptotically faster than any $\exp(-C\abs{S})$. We use the perturbation datum $(\mathcal{K}_r,J_r)$ to introduce the space $\mathcal{M}_r$ of solutions $u:\Sigma_r \rightarrow M$ of the pseudo-holomorhic equation
\begin{equation}\label{pseudo-holomorphic equation}
    \begin{cases}
        (du-Y_r)^{0,1}_{J_r} = 0,&\\
        E(u) < +\infty,&\\
        u(\epsilon_0(s,t)) \subseteq X_{\sigma}& \quad\text{for } s\ll 0. 
    \end{cases}
\end{equation}
Here, $Y_r = X_{h(\rho) + F_r}\otimes \gamma_r$ is the $\omega$-dual to $-d_M\mathcal{K}_r$. It is a $1$-form on $\Sigma_r$ with values in the space of Hamiltonian vector fields. Moreover, the energy of $u$ is
\begin{equation}
    E(u) \defeq \frac{1}{2}\int_{\Sigma_r} \norm{du-Y_r}^2_{J_r}.
\end{equation}

\begin{lemma}\label{integrated maximum principle}
    If a solution $u$ to \eqref{pseudo-holomorphic equation} is disjoint from $D$, then $u(\Sigma_r)\subseteq X_{\sigma}$.
\end{lemma}

\begin{proof}
    This is a version of the integrated maximum principle of Abouzaid-Seidel, the proof is identical to \cite{ES1}.
\end{proof}

Suppose that we have a collection of $1$-periodic orbits $\bx = (x_i)_{i\in Z^c(\bn)}$ such that $x_0$ is not a $D$-orbit, and let $\mathbf{m} = (m_i)_{i\in Z(\bn)}$ be a tuple of `tangency vectors' $m_i \in (\Z_{\ge 0})^N$; here $Z(\bn) = \{i \in \{0,1,\ldots,d\}\ |\ n_i = 0\}$. 
Define the moduli space $\mapsms^{\al}_{d,\bp,\bn,\mathbf{m}}$ of pairs $(r,u)$ such that $u$ is a solution to the pseudo-holomorphic curve equation \eqref{pseudo-holomorphic equation} and
\begin{equation}
    \begin{cases}
        \lim_{s\to \pm \infty} \epsilon_i^*u(s,t) = x_i(t)&\quad \text{for}\ i\in Z^c(\bn)\\
        j_{m_{ik}-1,z_i}^{D_k}(u) = 0 &\quad \text{for}\ i\in Z(\bn).
    \end{cases}
\end{equation}
The second equation requires that $u$ is tangent to $D_k$ at $z=z_i$ to order $m_{ik}$, \cite[\S 6]{Cieliebak-Mohnke}. 

The virtual dimension of the moduli space (i.e., the index of the associated Fredholm operator) may be computed to be
\begin{equation}\label{eq: dimension formula 1}
        \text{vdim}_u \mapsms^{\al}_{d,\bp,\bn,\mathbf{m}} = \dim \domms^{\al}_{d,\bp,\bn} + \abs{x_0} - \sum_{i,n_i>0} \abs{x_i} - 2\sum_{i\in Z(\bn)} \abs{m_i},
    \end{equation}  
where we write $\abs{m_i} = \sum_{k=1}^N m_{ik}$ for any tangency vector $m_i = (m_{ik}) \in (\Z_{\ge 0})^N$, and the $1$-periodic orbits $(x_i)_{i\in Z^c(\bn)}$ are given caps $(u_i)_{i\in Z^c(\bn)}$ such that
\begin{equation}
    u_0 = -u\#(u_i\ | \ i\in Z^c(\bn)).
\end{equation}

The following is the main regularity result we will need:

\begin{lemma}\label{lemma: transversality}
    There exists a choice of perturbation data $\mathscr{F}_r$ satisfying \textbf{(Consistency)}, \textbf{(Forgetfulness)}, and \textbf{(Equivariance)}, such that the moduli spaces $\mapsms^{\al}_{d,\bp,\bn,\mathbf{m}}(\bx)$ are all Fredholm regular, with the sole exception of the case $\abs{Z^c(\bn)}=1$, $x_1=x_0$, $F=\emptyset$, $m_i = 0$ for all $i$, which consists entirely of trivial solutions. 
    
    Moreover, we may choose the perturbation data such that if $(r,u)$ is a non-trivial solution belonging to a $0$- or $1$-dimensional component of $\mapsms^{\al}_{d,\bp,\bn,\mathbf{m}}(\bx)$, and $i \in Z(\bn)$ is a weight-zero marked point with $u(z_i) \notin D$, then $u(z_i)$ does not lie on any $J_{z_i}$-holomorphic sphere of Chern number $1$. 
\end{lemma}
\begin{proof}
The construction is inductive with respect to the partial order generated by the relations $(d',\bp',\bn') \le (d,\bp,\bn)$ if $\domms_{d',\bp',\bn'}$ is a factor in a stratum of $\overline{\domms}_{d,\bp,\bn}$, or if $(d',\bp',\bn')$ is obtained by removing some inputs of weight zero which do not carry sprinkles. One easily verifies that this is a partial order.

Assume that regular perturbation data satisfying the desired conditions have been constructed over all moduli spaces $\domms_{d',\bp',\bn'}$ with $(d',\bp',\bn') < (d,\bp,\bn)$. If $(d,\bp,\bn)$ includes no inputs of weight zero without sprinkles, then the perturbation data are specified over the boundary by the \textbf{(Consistency)} condition and our inductive choices, and we may extend them over the interior so as to satisfy the \textbf{(Equivariance)} condition. 
This extension determines the perturbation data over the moduli spaces $\domms_{\tilde d,\tilde \bp,\tilde \bn}$ with $(\tilde d^{\red},\tilde \bp^\red,\tilde \bn^\red) = (d,\bp,\bn)$, in accordance with the \textbf{(Forgetfulness)} condition; these perturbation data satisfy \textbf{(Consistency)} and \textbf{(Equivariance)} automatically, and satisfy the required compatibility with the choice of subclosed $1$-forms and cylindrical ends \eqref{eq:KrFr}, \eqref{eq:Kr_cyl}, \eqref{eq:Jr_cyl} by the compatibility of these with forgetful maps. 

Using standard transversality techniques as in \cite[Lemma 6.7]{Cieliebak-Mohnke}, for a generic choice of such perturbation data, the moduli spaces $\mapsms^\al_{\tilde d,\tilde \bp,\tilde \bn, \mathbf{m}}(\bx)$ are regular for all $(\tilde d,\tilde \bp,\tilde \bn,\mathbf{m})$ such that $(\tilde d^{\red},\tilde \bp^\red,\tilde \bn^\red) = (d,\bp,\bn)$. 
Because $\Sym(d,\bp)$ acts freely on $\domms_{d,\bp,\bn}$, we may arrange this transversality compatibly with the \textbf{(Equivariance)} condition, cf. \cite[Lemma 5.13]{perutz2022constructing}. 
We may rule out so-called trivial solutions as in \cite[Lemma 4.4]{ES1}, making use of the fact that we arranged for all $SH$-type orbits to be disjoint in Lemma \ref{lemma: Hamiltonians}, and $x_0$ is not a $D$-orbit by assumption. 

To arrange the condition about non-intersection of curves with Chern number $1$ spheres, we consider two cases. If $i$ does not carry a sprinkle, then by the \textbf{(Forgetfulness)} property there exists a forgetful map 
$$\mapsms^{\al}_{d,\bp,\bn,\mathbf{m}}(\bx) \to \mapsms^{\al}_{d-1,\bp',\bn',\mathbf{m}'}(\bx)$$
which forgets the $i$th marked point. We have 
$$ \mathrm{vdim}\mapsms^{\al}_{d-1,\bp',\bn',\mathbf{m}'}(\bx) = \mathrm{vdim} \mapsms^{\al}_{d,\bp,\bn,\mathbf{m}}(\bx) - 2 < 0,$$
so the target of the forgetful map is empty; hence the source is also empty. 
On the other hand, if $i$ does carry a sprinkle, we use the fact that Chern number $1$ spheres sweep out a set of codimension $2$, whereas the evaluation map at $z_i$ sweeps out a set of dimension $\le 1$, hence they are generically disjoint. (Note that we really do need two arguments here: the first argument doesn't cover the second case, as there is no forgetful map forgetting a marked point carrying a sprinkle; and the second argument doesn't cover the first case due to the requirement that the perturbation data be pulled back via the forgetful map.)

We conclude by observing that in the special case of those moduli spaces such that forgetting all forgettable marked points gives a Floer cylinder, the regularity hypotheses are achieved by an appropriate choice of the data $J_n$, following the same argument above.
\end{proof}

Moreover, due to the curvature estimate $R(\mathcal{K}_r)\leq C_{\bn}d\gamma_r$, the moduli space $\mapsms^{\al}_{d,\bp,\bn,\mathbf{m}}(\bx)$ is empty unless
\begin{equation}
    A(x_0) \geq \sum_{i,n_i>0} A(x_i) + C_{\bn}\abs{F}.
\end{equation}

\subsection{Compactness}

The study of compactness is tricky because we only have transversality of moduli spaces $\mapsms^{\al}_{d,\bp,\bn,\mathbf{m}}(\bx)$ where the output $x_0$ is not a $D$-orbit. 
We will show that nevertheless, we have compactness of our moduli spaces under certain assumptions on the choice of orbits $\bx$.

Suppose that $\abs{m_i} = 1$ for all $i$ and that $\bx$ consists of $1$-periodic orbits such that for all $j\in Z^c(\bn)$,
\begin{equation}\label{F-filtration assumption}
   \mathscr{P}(x_j)>1 \quad \text{if } j>0,\quad\text{where}\quad \mathscr{P}(x) = Cn_x + \kappa^{-1}A(x) - \abs{x}.
\end{equation}
This assumption ensures in particular that the $1$-periodic orbits $x_j$ are in $X_{\sigma}$, by Lemma \ref{lemma: D-orbits are P-negative}. Assume further that
\begin{equation}\label{dimension assumption}
    \dim \domms^{\al}_{d,\bp,\bn} + \abs{x_0} - \sum_{i,n_i>0} \abs{x_i} - 2\abs{Z(\bn)} \leq 1.
\end{equation}

Let $u_{\infty}$ be a broken pseudo-holomorphic curve in the Gromov compactification which is modelled on a bubble tree $T$. The tree $T$ inherits from $(d,\bp,\bn)$ a tuple $(d_v,\bp_v,\bn_v)$ for each one of its vertices as described in \cite[Section 2.4]{ES1}. The tree $T$ is oriented from the root (denoted by $\rt$) to the leaves which are labelled by $\{1,\dots,d\}$. For each $v\in V(T)$, we denote by $e_{in}(v)$ its incoming edge (in the orientation root-to-leaves), $E_{\out}(v)$ the set of its outgoing edges, and $u_v$ the component of $u_{\infty}$ supported on $v$. 
In particular, spherical bubbles correspond to vertices $v$ such that $n_{e_{in}(v)} = 0$. Each internal edge $e$ of $T$ has a corresponding $1$-periodic Hamiltonian orbit $x_e$ for the Hamiltonian $H_{n_e}$, provided that $n_e>0$. The caps given to $x_e$ are obtained by adding the caps on the leaves descending from $e$, to the classes of curves $u_v$ associated with vertices $v$ descending from $e$. In the case when $n_e = 0$, we still use $x_e$ as a formal symbol and give it degree $2$ and action $0$. 

There are two potential compactness issues from which $u_{\infty}$ could suffer:
\begin{itemize}
    \item[(1)] It could have components limiting to a $D$-orbit.
    \item[(2)] It could have spherical bubbles.
\end{itemize}
We show that under assumptions \eqref{F-filtration assumption} and \eqref{dimension assumption}, neither issue occurs. A capped $1$-periodic orbit $x$ is said to be $\mathscr{P}$-negative if
\begin{equation}
    \mathscr{P}(x)<0.
\end{equation}
For example, Lemma \ref{lemma: D-orbits are P-negative} states that all $D$-orbits are $\mathscr{P}$-negative.

\begin{remark}
    The assumption \eqref{eq: C>2}, which says roughly that the slope of our Hamiltonian $h(\rho)$ should be sufficiently large near the divisor, is necessary to rule out the possibility of $u_\infty$ limiting to a $D$-orbit. Indeed it will be used in the proof of Lemma \ref{lemma: no D-orbits} below, but let us first give a more enlightening explanation of why it is necessary. Consider the framed cap from Example \ref{framed cap}, with $\mathbf{m} = (m_{1k}) = (\delta_{jk})$. Then the moduli space $\mapsms^{\al}_{d,\bp,\bn,\mathbf{m}}(\bx)$ can be interpreted as the leading-order part of the Piunikhin--Salamon--Schwarz isomorphism \cite{PSS}, applied to the Poincar\'e dual of the divisor $D_j$. When the slope of our Hamiltonian is small, the PSS map sends this class to a generator of Floer cohomology corresponding to a Morse critical point on $D_j$; in particular, a $D$-orbit. This cap could appear as a component $u_\infty$ in a Gromov limit. 
    However, as the slope of our Hamiltonian increases, the constant orbit bifurcates into a non-constant loop around $D_j$ (in particular, an $SH$-orbit); together with a pair of cancelling orbits, one of $SH$-type and one of $D$-type. 
    The cap $u_\infty$ now limits to an $SH$-orbit, rather than a $D$-orbit. 
\end{remark}

\begin{figure}
    \centering
    \begin{tikzpicture}[x=0.70pt,y=0.75pt,yscale=-1,xscale=1]
    %uncomment if require: \path (0,483); %set diagram left start at 0, and has height of 483
    %Curve Lines [id:da6303025811365792] 
    \draw    (213.89,55.17) .. controls (198.47,55.9) and (199.09,113.94) .. (215.12,114.68) ;
    %Curve Lines [id:da22793877190122558] 
    \draw    (35.05,55.17) .. controls (19.63,55.9) and (20.25,113.94) .. (36.28,114.68) ;
    %Straight Lines [id:da21737679151116196] 
    \draw    (35.05,55.17) -- (213.89,55.17) ;
    %Straight Lines [id:da11063236049264602] 
    \draw    (36.28,114.68) -- (215.12,114.68) ;
    %Curve Lines [id:da665224566289035] 
    \draw    (213.89,55.17) .. controls (226.84,55.17) and (227.46,115.41) .. (215.12,114.68) ;
    %Curve Lines [id:da731532019125158] 
    \draw  [dash pattern={on 4.5pt off 4.5pt}]  (35.05,55.17) .. controls (48,55.17) and (48.61,115.41) .. (36.28,114.68) ;
    \draw   (122.58,53) -- (126.35,57.49)(126.35,53) -- (122.58,57.49) ;
    \draw   (110.25,53) -- (114.02,57.49)(114.02,53) -- (110.25,57.49) ;
    \draw   (87.43,112.51) -- (91.2,117)(91.2,112.51) -- (87.43,117) ;
    %Curve Lines [id:da9278016638215127] 
    \draw    (24.56,83.82) .. controls (62.18,83.08) and (88.7,89.7) .. (89.32,115.41) ;
    %Shape: Ellipse [id:dp49160550450166807] 
    \draw   (63.41,87.86) .. controls (63.41,86.44) and (64.38,85.29) .. (65.57,85.29) .. controls (66.76,85.29) and (67.73,86.44) .. (67.73,87.86) .. controls (67.73,89.28) and (66.76,90.43) .. (65.57,90.43) .. controls (64.38,90.43) and (63.41,89.28) .. (63.41,87.86) -- cycle ;
    %Curve Lines [id:da5912311304672211] 
    \draw    (535.73,54.52) .. controls (518.54,55.25) and (519.23,112.96) .. (537.1,113.69) ;
    %Curve Lines [id:da11663561536973388] 
    \draw    (336.33,54.52) .. controls (319.14,55.25) and (319.83,112.96) .. (337.71,113.69) ;
    %Straight Lines [id:da11420364293200347] 
    \draw    (336.33,54.52) -- (535.73,54.52) ;
    %Curve Lines [id:da018304069445419957] 
    \draw    (535.73,54.52) .. controls (550.17,54.52) and (550.85,114.42) .. (537.1,113.69) ;
    %Curve Lines [id:da5038467472578603] 
    \draw  [dash pattern={on 4.5pt off 4.5pt}]  (336.33,54.52) .. controls (350.77,54.52) and (351.46,114.42) .. (337.71,113.69) ;
    \draw   (420.18,52.37) -- (424.38,56.83)(424.38,52.37) -- (420.18,56.83) ;
    %Shape: Ellipse [id:dp3984929919447022] 
    \draw   (394.43,154.23) .. controls (394.43,152.81) and (395.51,151.67) .. (396.84,151.67) .. controls (398.17,151.67) and (399.24,152.81) .. (399.24,154.23) .. controls (399.24,155.64) and (398.17,156.78) .. (396.84,156.78) .. controls (395.51,156.78) and (394.43,155.64) .. (394.43,154.23) -- cycle ;
    %Shape: Ellipse [id:dp6877650319731947] 
    \draw   (405.09,36.26) .. controls (405.09,26.18) and (412.78,18) .. (422.28,18) .. controls (431.77,18) and (439.47,26.18) .. (439.47,36.26) .. controls (439.47,46.35) and (431.77,54.52) .. (422.28,54.52) .. controls (412.78,54.52) and (405.09,46.35) .. (405.09,36.26) -- cycle ;
    \draw   (402.99,33.38) -- (407.19,37.84)(407.19,33.38) -- (402.99,37.84) ;
    \draw   (437.37,33.38) -- (441.57,37.84)(441.57,33.38) -- (437.37,37.84) ;
    %Curve Lines [id:da45425104384943893] 
    \draw    (378.27,114.42) .. controls (378.27,127.56) and (414.71,129.03) .. (414.03,114.42) ;
    %Straight Lines [id:da4291854476594832] 
    \draw    (337.71,113.69) -- (378.27,114.42) ;
    %Straight Lines [id:da7644349939458448] 
    \draw    (414.03,114.42) -- (537.1,113.69) ;
    %Straight Lines [id:da9631242818711367] 
    \draw    (378.27,114.42) -- (378.27,149.48) ;
    %Straight Lines [id:da3221095644657488] 
    \draw    (414.03,114.42) -- (414.03,149.48) ;
    %Curve Lines [id:da31788962003988863] 
    \draw    (378.27,149.48) .. controls (378.96,180.89) and (414.03,179.43) .. (414.03,149.48) ;
    %Straight Lines [id:da8197646272005725] 
    \draw    (396.84,124.64) -- (396.84,151.67) -- (396.84,172.85) ;
    %Curve Lines [id:da3561694341663708] 
    \draw  [dash pattern={on 4.5pt off 4.5pt}]  (378.27,114.42) .. controls (378.27,103.46) and (412.65,102.73) .. (414.03,114.42) ;
    %Shape: Ellipse [id:dp38349611804768813] 
    \draw   (458.38,132.31) .. controls (458.38,122.23) and (466.07,114.05) .. (475.56,114.05) .. controls (485.06,114.05) and (492.75,122.23) .. (492.75,132.31) .. controls (492.75,142.4) and (485.06,150.57) .. (475.56,150.57) .. controls (466.07,150.57) and (458.38,142.4) .. (458.38,132.31) -- cycle ;
    \draw   (473.81,112.26) -- (478.01,116.73)(478.01,112.26) -- (473.81,116.73) ;
    \draw   (394.74,170.7) -- (398.94,175.16)(398.94,170.7) -- (394.74,175.16) ;
    %Straight Lines [id:da22446349478217176] 
    \draw    (242,84) -- (305,84) ;
    \draw [shift={(307,84)}, rotate = 180] [color={rgb, 255:red, 0; green, 0; blue, 0 }  ][line width=0.75]    (10.93,-3.29) .. controls (6.95,-1.4) and (3.31,-0.3) .. (0,0) .. controls (3.31,0.3) and (6.95,1.4) .. (10.93,3.29)   ;
    %Curve Lines [id:da6623833354211672] 
    \draw    (323.96,85.2) .. controls (365.9,84.47) and (396.15,99.08) .. (396.84,124.64) ;
    %Straight Lines [id:da3854208689754337] 
    \draw    (44,247) -- (81,248) ;
    %Straight Lines [id:da7857424414312604] 
    \draw    (81,248) -- (112,307) ;
    %Straight Lines [id:da05541271894897348] 
    \draw    (112,307) -- (197,318) ;
    %Straight Lines [id:da00021563797617996094] 
    \draw    (81,248) -- (151,275) ;
    %Straight Lines [id:da37830609949767524] 
    \draw    (81,248) -- (185,249) ;
    %Straight Lines [id:da5382298668563497] 
    \draw    (81,248) -- (129,214) ;
    %Straight Lines [id:da9886792255906207] 
    \draw    (129,214) -- (152,236) ;
    %Straight Lines [id:da8105741212090496] 
    \draw    (129,214) -- (157,214) ;
    %Shape: Circle [id:dp42031915739152903] 
    \draw  [color={rgb, 255:red, 0; green, 0; blue, 0 }  ,draw opacity=1 ][fill={rgb, 255:red, 0; green, 0; blue, 0 }  ,fill opacity=1 ] (125.5,214) .. controls (125.5,212.07) and (127.07,210.5) .. (129,210.5) .. controls (130.93,210.5) and (132.5,212.07) .. (132.5,214) .. controls (132.5,215.93) and (130.93,217.5) .. (129,217.5) .. controls (127.07,217.5) and (125.5,215.93) .. (125.5,214) -- cycle ;
    %Shape: Circle [id:dp045023719685508556] 
    \draw  [color={rgb, 255:red, 0; green, 0; blue, 0 }  ,draw opacity=1 ][fill={rgb, 255:red, 0; green, 0; blue, 0 }  ,fill opacity=1 ] (79,249) .. controls (79,247.07) and (80.57,245.5) .. (82.5,245.5) .. controls (84.43,245.5) and (86,247.07) .. (86,249) .. controls (86,250.93) and (84.43,252.5) .. (82.5,252.5) .. controls (80.57,252.5) and (79,250.93) .. (79,249) -- cycle ;
    %Shape: Circle [id:dp4381016358054015] 
    \draw  [color={rgb, 255:red, 0; green, 0; blue, 0 }  ,draw opacity=1 ][fill={rgb, 255:red, 0; green, 0; blue, 0 }  ,fill opacity=1 ] (107,306) .. controls (107,304.07) and (108.57,302.5) .. (110.5,302.5) .. controls (112.43,302.5) and (114,304.07) .. (114,306) .. controls (114,307.93) and (112.43,309.5) .. (110.5,309.5) .. controls (108.57,309.5) and (107,307.93) .. (107,306) -- cycle ;
    %Shape: Circle [id:dp2452783726273724] 
    \draw  [color={rgb, 255:red, 0; green, 0; blue, 0 }  ,draw opacity=1 ][fill={rgb, 255:red, 0; green, 0; blue, 0 }  ,fill opacity=1 ] (147.5,275) .. controls (147.5,273.07) and (149.07,271.5) .. (151,271.5) .. controls (152.93,271.5) and (154.5,273.07) .. (154.5,275) .. controls (154.5,276.93) and (152.93,278.5) .. (151,278.5) .. controls (149.07,278.5) and (147.5,276.93) .. (147.5,275) -- cycle ;
    %Straight Lines [id:da5702199804930244] 
    \draw    (351,249) -- (388,250) ;
    %Straight Lines [id:da8586808000018125] 
    \draw    (388,250) -- (419,309) ;
    %Straight Lines [id:da4852013839943352] 
    \draw    (419,309) -- (504,320) ;
    %Straight Lines [id:da8163539725799762] 
    \draw    (388,250) -- (458,277) ;
    %Straight Lines [id:da43830118933406514] 
    \draw    (388,250) -- (492,251) ;
    %Straight Lines [id:da44738946214851105] 
    \draw    (388,250) -- (436,216) ;
    %Shape: Circle [id:dp9714502492479566] 
    \draw  [color={rgb, 255:red, 0; green, 0; blue, 0 }  ,draw opacity=1 ][fill={rgb, 255:red, 0; green, 0; blue, 0 }  ,fill opacity=1 ] (386,251) .. controls (386,249.07) and (387.57,247.5) .. (389.5,247.5) .. controls (391.43,247.5) and (393,249.07) .. (393,251) .. controls (393,252.93) and (391.43,254.5) .. (389.5,254.5) .. controls (387.57,254.5) and (386,252.93) .. (386,251) -- cycle ;
    %Shape: Circle [id:dp2718671050398991] 
    \draw  [color={rgb, 255:red, 0; green, 0; blue, 0 }  ,draw opacity=1 ][fill={rgb, 255:red, 0; green, 0; blue, 0 }  ,fill opacity=1 ] (414,308) .. controls (414,306.07) and (415.57,304.5) .. (417.5,304.5) .. controls (419.43,304.5) and (421,306.07) .. (421,308) .. controls (421,309.93) and (419.43,311.5) .. (417.5,311.5) .. controls (415.57,311.5) and (414,309.93) .. (414,308) -- cycle ;
    %Straight Lines [id:da81898086966123] 
    \draw    (222,249) -- (324,249) ;
    \draw [shift={(326,249)}, rotate = 180] [color={rgb, 255:red, 0; green, 0; blue, 0 }  ][line width=0.75]    (10.93,-3.29) .. controls (6.95,-1.4) and (3.31,-0.3) .. (0,0) .. controls (3.31,0.3) and (6.95,1.4) .. (10.93,3.29)   ;

    % Text Node
    \draw (115,73) node [anchor=north west][inner sep=0.75pt]   [align=left] {$u$};
    % Text Node
    \draw (366,275) node [anchor=north west][inner sep=0.75pt]   [align=left] {$T'$};
    % Text Node
    \draw (55,278) node [anchor=north west][inner sep=0.75pt]   [align=left] {$T$};
    % Text Node
    \draw (420,73) node [anchor=north west][inner sep=0.75pt]   [align=left] {$u_{\infty}$};
    % Text Node
    \draw (247,65) node [anchor=north west][inner sep=0.75pt]   [align=left] {Gromov};
    % Text Node
    \draw (257,87) node [anchor=north west][inner sep=0.75pt]   [align=left] {limit};
    % Text Node
    \draw (248,225) node [anchor=north west][inner sep=0.75pt]   [align=left] {Pruning};
    \end{tikzpicture}
    \caption{Pruning weight $0$ internal edges of a tree.}
    \label{fig: gromov compactification}
\end{figure}

\begin{lemma}\label{lemma: no D-orbits}
    No component of $u_{\infty}$ limits to a $\mathscr{P}$-negative orbit.
\end{lemma}

\begin{proof}
    Let $T'$ be the subtree of $T$ which is obtained by pruning all weight $0$ internal edges (see Figure \ref{fig: gromov compactification}). The tree $T'$ supports the non-spherical components of $u_{\infty}$ so that $n_{e_{in}(v)}>0$ for all $v\in V(T')$. It has the same number of sprinkles as $T$, but a potentially different number of leaves $d'$. We note that it is possible to have $d'\geq d$ if $T$ contains domain-unstable spherical bubbles.
    
    Consider the subset $V_D\subseteq V(T')$ of vertices given by the property that $v\in V_D$ iff
    \begin{itemize}
        \item[(i)] $x_{e_{in}(v)}$ is a $\mathscr{P}$-negative orbit, and
        \item[(ii)] For all $v'\in [\rt,v)$, $x_{e_{in}(v')}$ is not a $\mathscr{P}$-negative orbit.
    \end{itemize}
    A vertex $v\in V(T')$ is one of three types:
    \begin{itemize}
        \item[-] Type 1: The interval $[\rt,v]$ is disjoint from $V_D$. Each weight $0$ edge $e\in E_{\out}(v)$ is attached to a spherical bubble tree $B_e$ ($B_e$ is empty if $e$ is a leaf), and contributes a (possibly $0$) tangency order $u_v\cdot_{z_e} D$ to $u_v$ in the dimension formula \eqref{eq: dimension formula 1}. 
        Since $x_{e_{in}(v)}$ is not a $\mathscr{P}$-negative orbit, as long as the curve $u_v$ is not a trivial solution, it is Fredholm regular so that
        \begin{align}\label{eq:type 1 reg}
            \dim \domms^{\al}_{d'_v,\bp'_v,\bn'_v} + \abs{x_{e_{in}(v)}} &\geq \sum_{\substack{e\in E_{\out}(v)\\n_e>0}}\abs{x_e}+\sum_{\substack{e\in E_{\out}(v)\\n_e=0}} 2(u_v\cdot_{z_e} D  + c_1(B_e)) \\
    \nonumber        & = \sum_{e \in E_{\out}(v)} \abs{x_e} + \sum_{\substack{e\in E_{\out}(v)\\n_e=0}} 2(u_v\cdot_{z_e} D  + c_1(B_e) - 1).
        \end{align}
        The terms $c_1(B_e)$ are due to the difference between the cap given to $x_{e_{in}(v)}$ from $T$, compared to the one it receives from $T'$. Observe that in all cases $u_v\cdot_{z_e} D  + c_1(B_e)\geq 1$. Therefore,
        \begin{equation}\label{eq:reg ineq}
            \dim \domms^{\al}_{d'_v,\bp'_v,\bn'_v} + \abs{x_{e_{in}(v)}}\geq \sum_{e\in E_{\out}(v)}\abs{x_e}.
        \end{equation} 
        We now address the case that $u_v$ is a trivial solution, so its image is contained in an $SH$-type orbit $x_{e_{in}(v)}$, there is a single edge $e' \in E_{\out}(v)$ with $n_e>0$, and $x_{e'} = x_{e_{in}(v)}$. 
        We then have
        \begin{equation}
            \abs{x_{e_{in}(v)}} = \abs{x_{e'}} + \sum_{\substack{e\in E_{\out}(v)\\n_e=0}} 2c_1(B_e),
        \end{equation}
        and $\dim \domms^\al_{d'_v,\bp'_v,\bn'_v} \ge 0$ (this is required for the curve to be stable), from which \eqref{eq:type 1 reg} follows. We have $c_1(B_e) \ge 1$ for all $e \in E_\out(v)$ with $n_e = 0$, so \eqref{eq:reg ineq} follows.
        
        \item[-] Type 2: $v\in V_D$, so that $x_{e_{in}(v)}$ is a $\mathscr{P}$-negative orbit. In this case,
        \begin{align*}
            \abs{x_{e_{in}(v)}}&\geq Cn_{x_{e_{in}(v)}} + \kappa^{-1}A(x_{e_{in}(v)})\\
            &\geq Cn_{x_{e_{in}(v)}} + \kappa^{-1}C_{\bn_v}\abs{F_v}+ \sum_{e\in E_{\out}(v)} \kappa^{-1}A(x_e) \\
            &= C\abs{F_v} + \kappa^{-1}C_{\bn_v}\abs{F_v} + \sum_{e\in E_{\out}(v)} Cn_{x_e}+ \kappa^{-1}A(x_e)
        \end{align*}
        We rewrite this inequality as
        \begin{equation}
            -(C+\kappa^{-1}C_{\bn_v})\abs{F_v} + \abs{x_{e_{in}(v)}} \geq \sum_{e\in E_{\out}(v)} Cn_{x_e}+ \kappa^{-1}A(x_e)
        \end{equation}
        \item[-] Type 3: The interval $[\rt,v)$ contains an element from $V_D$. In this case, we use the action inequality
        \begin{equation}
            -(C+\kappa^{-1}C_{\bn_v})\abs{F_v} + Cn_{x_{e_{in}(v)}} +\kappa^{-1}A(x_{e_{in}(v)}) \geq \sum_{e\in E_{\out}(v)} Cn_{x_e}+ \kappa^{-1}A(x_e).
        \end{equation}
    \end{itemize}
    We now add the previous inequalities across all $v\in V(T')$. The result is
    \begin{equation}\label{compactness inequality}
     \sum_{v\in I} \dim \domms^{\al}_{d'_v,\bp'_v,\bn'_v} - \sum_{v\notin I}(C+\kappa^{-1}C_{\bn_v})\abs{F_v}  + \abs{x_0} \geq \sum_{j\in J} \abs{x_j}  + \sum_{j\notin J}Cn_{x_j}+ \kappa^{-1}A(x_j),
    \end{equation}
    where $I\subseteq V(T')$ is the collection of vertices of type 1, and $J$ is the collection of leaves of $T'$ whose path to the root contains no element of $V_D$. 
    
    We rewrite the inequality \eqref{compactness inequality} as
    \begin{align*}
        \sum_{j\notin J} \mathscr{P}(x_j) 
        & \leq \sum_{v\in V(T')} \dim \domms^{\al}_{d'_v,\bp'_v,\bn'_v} + \abs{x_0}-\sum_{i=1}^{d'}\abs{x_i}\\
        & \quad\quad\quad - \sum_{v\notin I}(\dim\domms^{\al}_{d'_v,\bp'_v,\bn'_v} + (C+\kappa^{-1}C_{\bn_v})\abs{F_v}) \\
        & = \dim \domms^{\al}_{d',\bp',\bn'} - \abs{\iE(T')} + \abs{x_0}-\sum_{i=1}^{d'}\abs{x_i}\\ 
        &- \sum_{v\notin I}(2d_v-3+(1+C+\kappa^{-1}C_{\bn_v})\abs{F_v})\\
        &= \dim \domms^{\al}_{d,\bp,\bn} + \abs{x_0}-\sum_{i=1}^{d}\abs{x_i}\\ 
        &- \abs{V(T')} + 1
        - \sum_{v\notin I}(2d_v-3+(1+C+\kappa^{-1}C_{\bn_v})\abs{F_v})
    \end{align*}
    We now decompose the complement $J^c$ of $J$ into null-weighted leaves and positively-weighted leaves,
    \begin{equation}
        J^c_{0} = \{ j\notin J \ | \ n_j = 0 \} \quad\text{and}\quad J^c_{>0} = \{ j\notin J \ | \ n_j > 0 \}.
    \end{equation}
    
    Moreover, for each vertex $v\notin I$, we write $J^c_v$ for the set of leaves $j\notin J$ which are adjacent to $v$. Since $J^c$ is the union of all $(J^c_v)_{v\notin I}$, we can rewrite the previous inequality as
    \begin{align}\label{key inequality}
        \sum_{v\notin I} (2d_v-2+&(1+C+\kappa^{-1}C_{\bn_v})\abs{F_v}+ \sum_{j\in J^c_v} \mathscr{P}(x_j)) \\
        &\leq  \dim \domms^{\al}_{d,\bp,\bn}+\abs{x_0}-\sum_{i=1}^d \abs{x_i} + 1-\abs{V(T')}+\abs{V(T')\backslash I}\\
        &\leq 2 - \abs{I} \leq 1.
    \end{align}
    However, since $\mathscr{P}(x_j)=-2$ for all $j\in J^c_{0}$, we deduce that for all $v\notin I$,
    \begin{align*}
        2d_v-2+&(1+C+\kappa^{-1}C_{\bn_v})\abs{F_v}+ \sum_{j\in J^c_v} \mathscr{P}(x_j)\\
        &= 2(d_v-\abs{J^c_v\cap J^c_{0}})-2 + (1+C+\kappa^{-1}C_{\bn_v})\abs{F_v} + \sum_{j\in J^c_v\cap J^c_{>0}} \mathscr{P}(x_j)\notag\\
        &\geq \sum_{j\in J^c_v\cap J^c_{>0}} \mathscr{P}(x_j).\notag
    \end{align*}
    The last inequality is true because $2(d_v-\abs{J^c_v\cap J^c_{0}})\geq 2$, unless all of the outgoing vertices of $v$ are weight-zero leaves. In that case, $F_v$ is non-empty and $(1+C+\kappa^{-1}C_{\bn_v})\abs{F_v} > 3$, see \eqref{eq: C>2}.
    
    It follows that $\sum_{j\in J^c_{>0}} \mathscr{P}(x_j) \leq 1$ and hence $J^c_{>0}$ must be empty. Therefore, \eqref{key inequality} reduces to 
    \begin{equation}
       \sum_{v\notin I} 2(d_v-\abs{J^c_v\cap J^c_{0}})-2+(1+C+\kappa^{-1}C_{\bn_v})\abs{F_v} \leq 1
    \end{equation}
    
    Suppose now that $V_D$ is non-empty and let $v\in V_D$. Then $v\notin I$, and since $J^c_{>0}$ is empty, all the leaves flowing from $v$ must have weight $0$. It follows that there is at least one vertex $v'\notin I$ which flows from $v$ such that $\abs{F_{v'}}\geq 1$. Therefore, we get that $2(d_{v'}-\abs{J^c_{v'}\cap J^c_{0}})+(1+C+\kappa^{-1}C_{\bn_{v'}})\abs{F_{v'}} > 3$ for at least one $v'\notin I$. This is in contradiction with the inequality above. We conclude that $V_D$ must be empty.
\end{proof}

\begin{remark}\label{remark: compactness for $0$-dimensional moduli spaces}
    In the case of $0$-dimensional moduli spaces
    \begin{equation}
        \dim \domms^{\al}_{d,\bp,\bn} + \abs{x_0} - \sum_{i,n_i>0} \abs{x_i} - 2\abs{Z(\bn)} = 0,
    \end{equation}
    the same proof above works provided only $\mathscr{P}(x_j)>0$ for all $j\in Z^c(\bn)$.
\end{remark}

\begin{lemma}\label{lemma: no sphere bubbles}
    No component of $u_{\infty}$ is a spherical bubble.
\end{lemma}
\begin{proof}
    We use the same notation and terminology of the previous lemma. The tree $T'\subseteq T$ has the property that all of its internal edges satisfy $n_{e}>0$, and the $1$-periodic orbit $x_e$ is not a $\mathscr{P}$-negative orbit. In particular, by \eqref{eq:type 1 reg} we have
    \begin{equation}
        \dim \domms^{\al}_{d'_v,\bp'_v,\bn'_v} + \abs{x_{e_{in}(v)}}\geq \sum_{e\in E_{\out}(v)}\abs{x_e}+\sum_{\substack{e\in E_{\out}(v)\\n_e=0}} 2(u_v\cdot_{z_e} D  + c_1(B_e) - 1)
    \end{equation}
        for all $v\in V(T')$.  
        By summing over all vertices, we get
    \begin{equation}
        \dim \domms^{\al}_{d',\bp',\bn'} - \abs{\iE(T')} + \abs{x_0}\geq \sum_{j\in\lf(T')} \abs{x_j} + \sum_{\substack{j\in\lf(T')\\ n_j=0}} 2(u\cdot_{z_j} D + c_1(B_j) - 1),
    \end{equation}
    where $\lf(T')$ is the set of leaves of $T'$. Using \eqref{dimension assumption}, we deduce that
    \begin{equation}
        2\sum_{\substack{j\in\lf(T')\\ n_j=0}} (u\cdot_{z_j} D + c_1(B_j)-1) \leq 1-\abs{\iE(T')}.
    \end{equation}
    Recall that $u\cdot_{z_j} D + c_1(B_j)\geq 1$ for all $0$-weighted leaves $j$. It follows that $u\cdot_{z_j} D + c_1(B_j) = 1$ for all $0$-weighted leaves. This can only happen if $u\cdot_{z_j} D = 0$ and $B_j$ is a non-constant bubble tree of Chern number $1$. Such bubbles are ruled out for trivial solutions by our choice of Floer data $J_n$, and for non-trivial solutions by our choice of perturbation data, see Lemma \ref{lemma: transversality}.
\end{proof}

\section{$L_{\infty}$ operations}

In this section we use the moduli spaces constructed in the previous section to construct the algebraic structures needed for our main construction.

\subsection{Symplectic cohomology}

The sequence $\bH$ of Hamiltonians can be used to describe the symplectic cohomology of the Liouville domain $X_{\sigma}\subseteq X$. For each term $H_n$ in the basic sequence $\bH$, let $CF^*(X_{\sigma},H_n)$ be the $\mathbb{Z}$-graded vector space over $\bk$ which is freely generated by the $1$-periodic orbits of $H_n$ inside $X_{\sigma}$,
\begin{equation}\label{eq: Floer complex for SH}
    CF^*(X_{\sigma}, H_n) = \bigoplus_{\gamma\subseteq X_{\sigma}} \abs{o_{\gamma}}_{\bk}.
\end{equation}
In \eqref{eq: Floer complex for SH}, the $1$-periodic orbit $\gamma$ is required to be contractible inside of $M$, but we do not equip it with a cap. Its index $\abs{\gamma}_{SH}$ is computed using the trivialization of $K_{X_{\sigma}}^{\otimes 2}$ which is determined by the equation $2c_1(M) = \pd(D^{\blambda})$, where
\begin{equation}
    D^{\blambda} = \lambda_1D_1 + \dots + \lambda_nD_n.
\end{equation}
In \cite[Lemma 3.8]{BSV}, it is shown that
\begin{equation}\label{eq: SH-index}
    \abs{\gamma}_{SH} = \abs{(\gamma,u)} + u\cdot D^{\blambda},
\end{equation}
where $u$ is any cap for $\gamma$ inside $M$. 
Similarly, the $SH$-action of $\gamma$ is defined by
\begin{equation}
    A^{SH}_{H_n}(\gamma) = -\int_{S^1}\gamma^*\theta + \int_{S^1} H_n(\gamma(t))dt.
\end{equation}
It is related to our previous action functional \eqref{eq: action functional} by the formula
\begin{equation}
    A^{SH}_{H_n}(\gamma) = A_{H_n}(\gamma, u) + \kappa\cdot u\cdot D^{\blambda}.
\end{equation}
In \cite[\S 1.4]{BSV}, the authors introduce a \emph{fractional cap} which is written formally as
\begin{equation}
    u_{in} = u - u\cdot D^{\blambda},
\end{equation}
where $u$ is any cap for $\gamma$ in $M$. In terms of the fractional cap, we have
\begin{equation}
    \abs{\gamma}_{SH} = \abs{(\gamma,u_{in})}\quad\text{and}\quad A^{SH}_{H_n}(\gamma) = A_{H_n}(\gamma, u_{in}).
\end{equation}

For each $n\in \mathbb{N}$, we have a continuation chain map
\begin{equation}
    c: CF^*(X_{\sigma},H_n)\rightarrow CF^*(X_{\sigma},H_{n+1}).
\end{equation}
The homotopy direct limit of continuation maps is the chain complex
\begin{equation}\label{telescope complex}
    SC^*(X_{\sigma},\mathbf{H}) = \bigoplus_{n=1}^{\infty}CF^*(X_{\sigma},H_n)[t],
\end{equation}
where $t$ is a formal variable of degree $-1$ such that $t^2=0$. It is equipped with the differential
\begin{equation}
    d^{SC}(\gamma+t\gamma') = (-1)^{\abs{\gamma}}d\gamma + (-1)^{\abs{\gamma'}}(td\gamma' + c(\gamma') - \gamma').
\end{equation}
The cohomology of this complex is called symplectic cohomology and it is denoted by $SH^*(X_{\sigma})$. It only depends on the deformation type of the Liouville domain $X_{\sigma}$.

\begin{remark}
    When the weights $\lambda_j$ are not integers, but only rational numbers, the degree $\abs{\gamma}_{SH}$ may not be an integer, so we need to explain the meaning of the sign $(-1)^{\abs{\gamma}}$. In this case we actually equip our complex with a $(\Q \oplus \Z/2)$-grading, where the $\Q$-component is as defined above, and the $\Z/2$-component is the parity of the $\Z$-grading defined using any cap (which is independent of the choice of cap by \eqref{eq:ind_act_change}). Then if $\abs{\gamma} = (q,\sigma)$, we define $(-1)^{\abs{\gamma}}:=(-1)^\sigma$.
\end{remark}

\subsection{Partial $L_{\infty}$ operations}

Consider a triplet $(d,\bp,\bn)$, and a tuple $\bx = (x_0,\dots,x_d)$ such that for all $i \in \{0,\dots,d\}$:
\begin{itemize}
    \item[-] If $i\in Z^c(\bn)$, then $x_i$ is an $SH$-orbit of $H_{n_i}$.
    \item[-] If $i\in Z(\bn)$, then $x_i = x^{D_{j_i}}$ is a label for one of the components of $D$.  
\end{itemize}

We introduce notation
$$\mapsms^{\al}_{d,\bp,\bn}(\bx)\defeq \left\{ (r,u) \in \mapsms^{\al}_{d,\bp,\bn,\mathbf{m}}(\bx') \ \bigg|\ [u] \cdot D = \abs{Z(\bn)}\right\},$$
where $\bx'$ is obtained from $\bx$ by removing all $x_i$ with $i \in Z(\bn)$, and the tangency data for such $i$ are given by $m_{ik} = \delta_{k,j_i}$.

By Lemma \ref{lemma: transversality}, $\mapsms^{\al}_{d,\bp,\bn}(\bx)$ is Fredholm regular of dimension
\begin{equation} \label{eq: dimension formula}
    \dim \mapsms^{\al}_{d,\bp,\bn}(\bx) = \dim \domms^{\al}_{d,\bp,\bn} + \abs{x_0} - \sum_{i=1}^d \abs{x_i}.
\end{equation}

In our previous work on compactness, we used formal symbols $x_e$ associated to inputs or nodes of weight zero, and these were given index $2$ and action $0$. Morally, these generators can be thought of as constant orbits equipped with constant caps. In \eqref{eq: dimension formula} and afterwards however, we give the analogous generators $x^{D_j}$ their $SH$-index and $SH$-action, which are
\begin{equation}
    \abs{x^{D_j}} = 2-\lambda_j\quad\text{and}\quad A(x^{D_j}) = -\kappa\lambda_j.
\end{equation}
Morally, this corresponds to equipping these constant orbits with their fractional cap $u_{in}$.

Each isolated element $\bu = (r,u) \in \mapsms^{\al}_{d,\bp,\bn}(\bx)$ determines a linear isomorphism
\begin{equation}\label{eq: u contribution}
    \abs{o_{\mathbf{u}}}: t^{i_1}\abs{o_{x_1}}\otimes\dots\otimes t^{i_d}\abs{o_{x_d}} \rightarrow \abs{o_{x_0}}\ [3-2d],
\end{equation}
where $i_j = \abs{\bp^{-1}(j)}$ for each $j\in\{1,\dots, d\}$ (cf. \cite[Section 6]{ES1}). 
We remark that for the formal generators $x^{D_j}$, we define $\abs{o_{x^{D_j}}}_\bk \defeq \bk$ to be a trivial line.

By Lemma \ref{lemma: no D-orbits}, Lemma \ref{lemma: no sphere bubbles}, and Remark \ref{remark: compactness for $0$-dimensional moduli spaces}, the moduli spaces $\mapsms^{\al}_{d,\bp,\bn}(\bx)$ are compact in dimension $0$, provided that
\begin{equation}
    \mathscr{P}(x_j)>0 \quad \text{ for all }j\in Z^c(\bn).
\end{equation}

Therefore, we can add up the contributions of \eqref{eq: u contribution} to produce maps
\begin{equation}
    \tilde{\ell}^d_0: \mathfrak{g}_{>0}^{\otimes d}\rightarrow \bigoplus_{n=1}^{\infty} \mathscr{P}_{>0} CF^*(X_{\sigma},H_n),
\end{equation}
where for any real $a \in \mathbb{R}$,
\begin{equation}\label{eq: partial L-infty algebra}
    \mathfrak{g}_{>a} = \mathscr{P}_{>a} SC^*(X_{\sigma})\oplus \bigoplus_{j=1}^N \bk[t] \cdot \abs{o_{x^{D_j}}}_{\bk}.
\end{equation}

\begin{remark}\label{rmk:sphere_bubble_output}
    Recall that the moduli space is undefined in the case that all inputs have weight $0$ and no sprinkle. 
    In this case, we define the output to be $0$:
    $$\tilde \ell^d\left(x^{D_{j_1}},\ldots,x^{D_{j_d}}\right) = 0.$$
    Morally, this is the correct definition because these generators correspond to constant orbits, on which the loop rotation action is trivial and hence the $L_\infty$ operations vanish; but we will see the precise justification in Lemma \ref{lem:L infinity relations} below.
\end{remark}
    
These maps have unique $\partial_t$-linear extensions
\begin{equation}
    \tilde{\ell}^{d}:\mathfrak{g}_{>0}^{\otimes d}\rightarrow \mathfrak{g}_{>0}.
\end{equation}
See \cite{ES1} for more details on the construction of $\tilde{\ell}^{d}$. 

Finally, we set
\begin{equation}
    \ell^{d}(x_1,\dots,x_d) =
    \begin{cases*}
        \tilde{\ell}^d (x_1,\dots,x_d) & if $d\geq 2$\\
        \tilde{\ell}^1(x_1) + (-1)^{\abs{x_1}}\partial_t x_1 & if $d = 1$.
    \end{cases*}
\end{equation}
We note that the outputs of the operations $\ell^d$ are always linear combinations of $1$-periodic orbits, except when $d=1$.

\begin{lemma}
    \label{lem:L infinity relations}
    The $L_{\infty}$ equations 
\begin{equation}\label{eq: L-infinity equation}
    \sum_{j=1}^d\sum_{\sigma\in \Unsh(d,j)} (-1)^{\epsilon(\sigma;x_1,\ldots,x_d)}\ell^{d-j+1}(\ell^{j}(x_{\sigma(1)},\dots,x_{\sigma(j)}),x_{\sigma(j+1)},\dots,x_{\sigma(d)}) = 0
\end{equation}
hold whenever  $x_1,\dots,x_d \in \mathfrak{g}_{>1}$. 
Here $\Unsh(d,j)$ is the subgroup of $\mathfrak{S}_d$ of permutations of $\{1,\dots,d\}$ which satisfy
\begin{equation}
    \sigma(1)<\dots<\sigma(j)\quad\text{and}\quad \sigma(j+1)<\dots<\sigma(d),
\end{equation}
and
\begin{equation}\label{eq: epsilon sign}
    \epsilon(\sigma;x_1,\dots,x_d) = \sum_{\substack{i<j \\ \sigma(i)>\sigma(j)}} \abs{x_i} \cdot \abs{x_j}.   
\end{equation} 
\end{lemma}
\begin{proof}
    The argument is as in \cite{ES1}, using the compactness results of Lemma \ref{lemma: no D-orbits} and Lemma \ref{lemma: no sphere bubbles}. 
The key novelty, compared with the argument in \emph{op. cit.}, concerns boundary components where a sphere bubble carrying several marked points of weight zero and no sprinkle develops. 
Such boundary components do not appear in the boundary of a $1$-dimensional moduli space by Lemma \ref{lemma: no sphere bubbles}; this justifies the definition given in Remark \ref{rmk:sphere_bubble_output}, as it makes the corresponding terms in \eqref{eq: L-infinity equation} vanish. 
\end{proof}

We emphasize that $(\mathfrak{g}_{>1}, (\ell^d)_{d\geq 1})$ is not an $L_{\infty}$ algebra because it is not preserved by $\ell^1$. Instead, we have $\ell^1(\mathfrak{g}_{>1}) \subseteq \mathfrak{g}_{>0}$.
\begin{lemma}\label{lem:elld P}
    The operations $\ell^d$ satisfy the following property
    \begin{equation}
        \mathscr{P}(\ell^d(x_1,\dots,x_d)) \geq \sum_{i=1}^d \mathscr{P}(x_i) + 2d-3.
    \end{equation}
\end{lemma}
\begin{proof}
    Indeed, if a $0$-dimensional moduli space $\mapsms^{\al}_{d,\bp,\bn}(\bx)$ is non-empty, we must have
    \begin{align*}
        2d - 3 + \abs{F} + \abs{x_0} &= \sum_{i=1}^d \abs{x_i}\\
        -(C+\kappa^{-1}C_\bn)\abs{F} + Cn_0 +\kappa^{-1}A(x_0)&\geq \sum_{i=1}^d Cn_i + \kappa^{-1}A(x_i).
    \end{align*}
We subtract the first equation from the second inequality. The result is
\begin{equation}\label{eq: F-inequality}
    \mathscr{P}(x_0) \geq \sum_{i=1}^d \mathscr{P}(x_i) + 2d-3 + (1+C+\kappa^{-1}C_{\bn})\abs{F}.
\end{equation}
The lemma follows from the fact that $1+C+\kappa^{-1}C_{\bn}>0$.
\end{proof}

The operations $(\ell^d)$ are closely related to the $L_{\infty}$ structure on symplectic cohomology
\begin{equation}
    \ell_{SC}^{d}: SC^*(X_{\sigma})^{\otimes d} \rightarrow SC^*(X_{\sigma}) [3-2d]
\end{equation}
which we constructed in \cite{ES1}:

\begin{lemma}\label{lem:SC agree}
    Suppose that we have inputs $x_1,\dots,x_d \in \mathscr{P}_{>0} SC^*(X_{\sigma})$. Then,
    \begin{equation}
        \ell_{SC}^{d}(x_1,\dots,x_d) = 
        \begin{cases}
         \ell^d(x_1,\dots,x_d) \quad &\text{if } d\geq 2\\
         \ell^1(x_1) \quad &\text{if } d=1 \text{ and } \mathscr{P}(x_1)>1.
        \end{cases}
    \end{equation}
\end{lemma}
\begin{proof}
    If $x_0 \in \mathscr{P}_{>0} SC^*(X_\sigma)$, then we have
    $$\langle \ell^d(x_1,\ldots,x_d),x_0\rangle = \langle \ell_{SC}^d(x_1,\ldots,x_d),x_0\rangle$$
    as a direct consequence of the integrated maximum principle of Lemma \ref{integrated maximum principle}. 
    If $x_0 \notin \mathscr{P}_{>0} SC^*(X_\sigma)$, then $\langle \ell^d(x_1,\ldots,x_d),x_0\rangle = 0$ by definition. 
    On the other hand, if $\langle \ell_{SC}^d(x_1,\ldots,x_d),x_0\rangle \neq 0$ and either $d \ge 2$ or $d=1$ and $\mathscr{P}(x_1)>1$, then $\mathscr{P}(x_0) > 0$ by the analogue of Lemma \ref{lem:elld P} for $\ell^d_{SC}$. 
\end{proof}

\subsection{Maurer--Cartan theory}

Let $(R,Q_{\geq \bullet})$ be a graded filtered commutative $\bk$-algebra. This means that $R$ has a decomposition into pure degree components
\begin{equation}
    R = \bigoplus_{a \in \mathbb{Q}} R_a,\quad R_a\cdot R_b \subseteq R_{a+b}.
\end{equation}
Moreover, we have a decreasing collection of $\bk$-subspaces $Q_{\geq p}R \subseteq R$ for each $p\in \mathbb{Z}$ such that
\begin{equation}
    Q_{\geq p_1}R \cdot  Q_{\geq p_2}R \subseteq Q_{\geq p_1 + p_2}R.
\end{equation}
We assume that the $Q$-filtration is exhaustive and bounded below.

For each $\bk$-vector space $V$, the tensor product $V\otimes_{\bk} R$ inherits a filtration
\begin{equation}
    Q_{\geq p} (V\otimes_{\bk} R)\defeq  V\otimes_{\bk} Q_{\geq p} R.
\end{equation}
The completion with respect to this filtration is denoted $\overline{V\otimes_{\bk} R}^Q$.

If $V$ is a graded vector space over $\bk$, we define the degreewise completed tensor product as the graded $R$-module
\begin{equation}
    V\wotimes R = \bigoplus_{a \in \mathbb{Q}} \overline{(V\otimes_{\bk} R)_a}^Q.
\end{equation}

\begin{example}
    Let $R = \bk[q]$ be a polynomial ring such that $\abs{q}=1$, and let $V$ be a $\mathbb{Z}$-graded vector space. Then,
    \begin{equation}
        (V\wotimes R)_i = \left\{ \sum_{k= -i}^{\infty}  v_{-k}q^{i+k} \ \bigg| \  v_{-k} \in V_{-k} \right\}.
    \end{equation}
    In particular, if $V$ is concentrated in degrees above some constant $c\in \mathbb{Z}$, we simply have $V\wotimes R  = V\otimes R$.
\end{example}

Let $(\mathfrak{g},\ell^*)$ be an $L_{\infty}$ algebra over $\bk$. We linearly extend the $L_{\infty}$ operations on $\mathfrak{g}$ to obtain an $L_{\infty}$ structure on $\mathfrak{g}^R = \mathfrak{g}\wotimes R$. Define the Maurer--Cartan set
\begin{equation}
    MC(\mathfrak{g},R) = \left\{\alpha \in Q_{\geq 1}\mathfrak{g}^R \ | \ \deg(\alpha)=2\quad\text{and}\quad \mathcal{F}(\alpha) = 0\right\},
\end{equation}
where $\mathcal{F}(\alpha) \in \mathfrak{g}^R$ is the curvature term
\begin{equation}
    \mathcal{F}(\alpha) = \sum_{d=1}^{\infty} \frac{1}{d!} \ell^d(\alpha,\dots,\alpha).
\end{equation}

Elements $\alpha\in MC(\mathfrak{g},R)$ are used to define deformed $L_{\infty}$ operations,
\begin{equation}
    \ell_{\alpha}^d(x_1,\dots,x_d) = \sum_{k=0}^{\infty} \frac{1}{k!}\ell^{k+d}(\alpha^{\otimes k},x_1,\dots,x_d).
\end{equation}
It is straightforward to check that $(\ell^d_{\alpha})$ satisfy the $L_{\infty}$ equations on $\mathfrak{g}^R$ (see \cite[Proposition 4.4]{Getzler-Lie-theory}).

Consider the $L_{\infty}$ algebra given by
\begin{equation}
   \mathfrak{g}_{\tau} = \mathfrak{g}\otimes_{\bk}  \bk[\tau,d\tau],
\end{equation}
where $\bk[\tau,d\tau]$ is a graded commutative dg-algebra with generators
\begin{equation}
    \deg(\tau) = 0,\quad \deg(d\tau)=1,\quad\text{and}\quad d(\tau) = d\tau.
\end{equation}
The $L_{\infty}$ operations on $\mathfrak{g}_{\tau}$ are given by
\begin{equation}
    \underline{\ell}^d(x_1\otimes\theta_1,\dots,x_d\otimes\theta_d) =
    \begin{cases}
        (-1)^{\dagger}\ell^d(x_1,\dots,x_d)\otimes(\theta_1\dotsi\theta_d)\quad & \text{ if }  d\geq 2,  \\
        \ell^1(x_1)\otimes \theta_1 + (-1)^{\abs{x_1}}x_1\otimes d\theta_1 \quad & \text{ if }  d=1,
    \end{cases}
\end{equation}
where the missing sign is $\dagger = \sum_{k<l} \abs{x_k}\abs{\theta_l}$. The $L_{\infty}$ equations for $(\underline{\ell}^*)$ follow from a straightforward computation using the $L_{\infty}$ equations for $(\ell^*)$, together with the following sign identity.
\begin{lemma}
    With each $\boldsymbol{y} = (x_1\otimes\theta_1,\dots,x_d\otimes\theta_d)$, we associate the sign
    \begin{equation}
        h(\boldsymbol{y}) = \sum_{k<l} \abs{x_k}\abs{\theta_l} \quad \mod 2.
    \end{equation}
    It is related to the Koszul signs defined in \eqref{eq: epsilon sign} by
    \begin{equation}
        \epsilon(\sigma;\boldsymbol{y}) - \epsilon(\sigma;\bx)-\epsilon(\sigma;\boldsymbol{\theta}) = h(\sigma\cdot \boldsymbol{y}) - h(\boldsymbol{y}) \quad \mod 2,
    \end{equation}
    where $\bx = (x_1,\dots,x_d)$, and $\boldsymbol{\theta} = (\theta_1,\dots,\theta_d)$.
\end{lemma}
\begin{proof}
    Observe that both sides of the equation satisfy the equation
    \begin{equation}
        \psi(\sigma_1\sigma_2;\bx) \equiv \psi(\sigma_1;\sigma_2\bx)+\psi(\sigma_2;\bx) \quad \mod 2.
    \end{equation}
    Therefore, it suffices to check that they agree for transpositions of the type $\sigma = (i\ i+1)$. In this case, both sides reduce to $\abs{x_i}\abs{\theta_{i+1}} + \abs{x_{i+1}}\abs{\theta_{i}}$.
\end{proof}

The $L_{\infty}$ algebra $(\mathfrak{g}_{\tau},\underline{\ell}^d)$ comes with evaluation maps for any $c\in \bk$,
\begin{equation}
    ev_{c}:\mathfrak{g}_{\tau} \rightarrow \mathfrak{g}: \quad\quad ev_c(\tau) = c \ \text{  and  }\ ev_c(d\tau) = 0,
\end{equation}
which are clearly $L_{\infty}$ homomorphisms, and in fact quasi-isomorphisms, see Lemma \ref{lemma: qis} below.

Two Maurer--Cartan elements $\alpha_0,\alpha_1\in MC(\mathfrak{g},R)$ are said to be gauge-equivalent if there is a Maurer--Cartan element $\alpha_\tau \in MC(\mathfrak{g}_{\tau},R)$ such that $ev_0(\alpha_{\tau}) = \alpha_0$ and $ev_1(\alpha_{\tau}) = \alpha_1$. We think of $\alpha_{\tau}$ as a homotopy between the two Maurer--Cartan elements. With that in mind, the following result can be thought of as a path lifting property for Maurer--Cartan elements.

\begin{lemma}\label{lemma: path extension for MC elements}
    Suppose we have $\alpha_0\in MC(\mathfrak{g},R)$ and $\gamma_{\tau} \in (\mathfrak{g}^R)^1\wotimes\bk[\tau]$. There is a unique $c_{\tau} \in (\mathfrak{g}^{R})^2\wotimes \bk[\tau]$ such that $\alpha_{\tau} = c_{\tau} + \gamma_{\tau}d\tau$ is a Maurer--Cartan element, and $ev_0(\alpha_{\tau}) = \alpha_0$.
\end{lemma}

\begin{proof}
 The Maurer--Cartan equation for $\alpha_{\tau}$ in $\mathfrak{g}^{R}_{\tau}$ decomposes into
 \begin{align}
    &\frac{d c_{\tau}}{d\tau} = - \sum_{d\geq 0} \frac{1}{d!}\ell^{1+d}(\gamma_{\tau},c_{\tau},\dotsi,c_{\tau}),\quad c_0 = \alpha_0. \label{eq: ODE}\\
    &\sum_{d\geq 1} \frac{1}{d!}\ell^d(c_{\tau},\dotsi,c_{\tau}) = 0.\label{eq: MC equation}
 \end{align}
 In the equations above, we clarify that unlike $\underline{\ell}^1(c_\tau)$, the expression $\ell^1(c_\tau)$ is an extension of $\ell^1$ to $\mathfrak{g}\wotimes \bk[\tau]$ which is linear in $\tau$.

 The ODE \eqref{eq: ODE} can be solved order-by-order in $\tau$, and it remains to check that the solution satisfies \eqref{eq: MC equation}. Indeed,
 \begin{align*}
    \frac{d}{d\tau}\mathcal{F}(c_{\tau})
    &= -\sum_{d,e\geq 0} \frac{1}{d!e!}\ell^{1+d}(\ell^{1+e}(\gamma_{\tau},c_{\tau},\dots,c_{\tau}),c_{\tau},\dots,c_{\tau})\\
    &= \sum_{d,e\geq 0} \ell^{1+d}(\ell^{1+e}(c_{\tau},\dots,c_{\tau}),\gamma_{\tau},c_{\tau},\dots,c_{\tau})\quad\quad(L_{\infty}-\text{equations})\\
    &= \sum_{d\geq 0} \frac{1}{d!}\ell^{1+d}(\mathcal{F}(c_{\tau}),\gamma_{\tau},c_{\tau},\dots,c_{\tau}).
 \end{align*}
 This is a first-order ODE in $\tau$ with initial condition $\mathcal{F}(c_{0}) = 0$. It follows that $\mathcal{F}(c_{\tau})$ vanishes identically as desired.
\end{proof}

\begin{lemma}\label{lemma: qis}
    If $\alpha_0$ and $\alpha_1$ are gauge-equivalent, then the deformed $L_{\infty}$ algebras $(\mathfrak{g},\ell^*_{\alpha_0})$ and $(\mathfrak{g},\ell^*_{\alpha_1})$ are quasi-isomorphic.
\end{lemma}
\begin{proof}
    It suffices to show that for any $c\in \bk$, the evaluation map
    \begin{equation}\label{eq: evaluation is qis}
        ev_c: (\mathfrak{g}^R_{\tau},\underline{\ell}^1_{\alpha_{\tau}})\rightarrow (\mathfrak{g}^R,\ell^1_{\alpha_c})
    \end{equation} 
    is a quasi-isomorphism. Indeed, the chain complexes in \eqref{eq: evaluation is qis} come with exhaustive and bounded below $Q$-filtrations which are respected by the map $ev_c$. Because $\alpha_{\tau}\in Q_{\geq 1}\mathfrak{g}^R_{\tau}$, the induced map at the $E_1$-page is
    \begin{equation}
        (Gr(\mathfrak{g}^R),\ell^1)\otimes (\bk[\tau,d\tau],d) \xrightarrow[]{\id\otimes ev_c}(Gr(\mathfrak{g}^R),\ell^1)\otimes(\bk,0),
    \end{equation}
    where $Gr(\mathfrak{g}^R)$ is the associated graded of $\mathfrak{g}^R$ with respect to the $Q$-filtration. The latter is a quasi-isomorphism because $ev_c:(\bk[\tau,d\tau],d)\rightarrow (\bk,0)$ is a quasi-isomorphism. By classical spectral sequence convergence  and comparison theorems (see \cite[Theorems 5.2.12 and 5.5.1]{Weibel}), we deduce that \eqref{eq: evaluation is qis} is a quasi-isomorphism.
\end{proof}

\subsection{Maurer--Cartan construction}

We now carry out the constructions of the previous section in our geometric setup. The (partial) $L_{\infty}$ algebra of interest is $\mathfrak{g}_{>0}$ as defined in \eqref{eq: partial L-infty algebra}, and we study $L_{\infty}$ deformations over the filtered $\bk$-algebra from the Introduction:
\begin{align*}\label{eq: ring R+}
    &R^+ = \bk[q_1,\dots,q_N],\quad \deg(q_j) = \lambda_j.\\
    &Q_{\geq p}R^+ = \bigoplus_{i\geq p} R^+_i,\quad p\in \mathbb{Z}.
\end{align*}

The partial $L_{\infty}$ algebra $\mathfrak{g}^{R^+}_{>0}$ has a tautological Maurer--Cartan element
\begin{equation}
    \alpha_0 = \bq\cdot \bx^D \defeq \sum_{j=1}^N q_jx^{D_j}
\end{equation}
(it satisfies the Maurer--Cartan equation because all terms vanish, see Remark \ref{rmk:sphere_bubble_output}).

Following the proof of Lemma \ref{lemma: path extension for MC elements}, we can homotope $\alpha_0$ through the constant path $\gamma_{\tau} = t\cdot\bq\cdot \bx^D$ by solving the ODE
\begin{equation}\label{eq: ctau}
    c_0 = \alpha_0,\quad \frac{d}{d\tau}c_\tau = \sum_{d\geq 0} \frac{1}{d!}\ell^{1+d}(\gamma,c_{\tau},\dots,c_{\tau}).
\end{equation}
This ODE can be solved order-by-order in $\tau$. However, since $\mathfrak{g}_{>0}$ is not exactly an $L_{\infty}$ algebra, more care is neeeded in checking that $\alpha_{\tau} = c_{\tau} + \gamma_{\tau}d\tau$ satisfies the Maurer--Cartan equations for $(\mathfrak{g}^{R^+}_{\tau,>0},\underline{\ell}^*)$.

\begin{lemma}
    We have  $c_{\tau}\in \mathfrak{g}^{R^+}_{\tau,>1}$.
\end{lemma}
\begin{proof}
    We start by looking at the first order term $c_1$ from the power series expansion $c_{\tau} = c_0 + c_1\tau + \dotsi$ of $c_{\tau}$. It is computed from the formula
    \begin{align}\label{eq: first order term}
        c_1 
        &= \sum_{d\geq 0} \frac{1}{d!}\ell^{1+d}(\gamma,\bq\cdot \bx^D,\dots,\bq\cdot \bx^D)\\
        &= -\bq\cdot \bx^D + \sum_{d\geq 0} \frac{1}{d!}\tilde{\ell}^{1+d}(\gamma,\bq\cdot \bx^D,\dots,\bq\cdot \bx^D).\notag
    \end{align}
    Let $x_0$ be an orbit contributing to $\tilde{\ell}^{1+d}(\gamma,\bq\cdot \bx^D,\dots,\bq\cdot \bx^D)$. By \eqref{eq: F-inequality} and \eqref{eq: C>2},
    \begin{equation}
        \mathscr{P}(x_0) \geq -1 - 2d + 2(1+d)-3 + (1+C+\kappa^{-1}C_{(1,0,\dots,0)})> 1.
    \end{equation}
    It follows that $c_1\in \mathfrak{g}^{R^+}_{\tau,>1}$. For higher order terms, we proceed by induction. First, note that because of \eqref{eq: first order term}, we can write
    \begin{equation}
        c_{\tau} = (1-\tau)\bq\cdot \bx^D + \overline{c}_{\tau},
    \end{equation}
    where $\overline{c}_{\tau}\in \mathscr{P}_{>0}SC^*(X_{\sigma})\otimes \bk[\tau] \wotimes R^+$ is a combination of $1$-periodic orbits, and a solution to the ODE
    \begin{equation}
        \overline{c}_{0}=0, \quad \frac{d}{d\tau}\overline{c}_{\tau} = \sum_{d\geq 0 } \frac{1}{d!}\tilde{\ell}^{1+d}(\gamma,(1-\tau)\bq\cdot \bx^D + \overline{c}_{\tau},\dots,(1-\tau)\bq\cdot \bx^D + \overline{c}_{\tau}).
    \end{equation}
    Let $i>1$ be an integer. Then $c_i = \overline{c}_i$ is computed from operations
    \begin{equation}
        \tilde{\ell}^{1+j+k}(\gamma,c_{i_1},\dots,c_{i_j},\bq\cdot \bx^D,\dots,\bq\cdot \bx^D),
    \end{equation}
    where $i_1,\dots,i_j < i$. If $x_0$ is an output contributing to this operation, then
    \begin{equation}
        \mathscr{P}(x_0) > -1 + j + (-2)\cdot k + 2(1+j+k) - 3 = 3j-2.
    \end{equation}
    Therefore, $\mathscr{P}(x_0)>1$ unless $j=0$. But the case of $j=0$ follows by the same argument (using \eqref{eq: F-inequality} and \eqref{eq: C>2}) presented above for $c_1$.
\end{proof}

\begin{corollary}
    The element $\alpha_{\tau} = c_{\tau} + \gamma_{\tau} d\tau\in \mathfrak{g}^{R^+}_{\tau,>1}$ satisfies the Maurer--Cartan equation. In particular, $\alpha_{1} \in SC^*(X_{\sigma})\widehat{\otimes} R^+$ satisfies the Maurer--Cartan equation. 
    The Maurer--Cartan element $\alpha$ from Construction \ref{const} is defined to be $\alpha_1$.
\end{corollary}
\begin{proof}
    Since $\gamma_{\tau},c_{\tau} \in \mathfrak{g}^{R^+}_{>1}\wotimes\bk[\tau]$, the $L_{\infty}$ equations needed for the proof of the Maurer--Cartan equation in Lemma \ref{lemma: path extension for MC elements} hold, so that the same argument works here as well.
\end{proof}

\subsection{Maurer--Cartan deformation}

For any $1$-periodic orbit $x$, define
\begin{equation}
    \mathcal{F}(x) = C\cdot n_x  + \kappa^{-1}A(x).
\end{equation}
It was shown in \cite[Proposition 1.14]{BSV} that the inclusions $\mathcal{F}^X_{\geq p} \subseteq SC^*(X_{\sigma})$ are quasi-isomorphisms, where 
\begin{equation}
    \mathcal{F}^X_{\geq p} = \mathcal{F}_{\geq p}SC^*(X_{\sigma}),\quad p\in\mathbb{R}.
\end{equation}

Let $\alpha_{\tau} \in \mathfrak{g}^{R^+}_{\tau,>1}$ be the Maurer--Cartan element previously constructed.
\begin{lemma}
    The subspace $\mathcal{F}_{\geq p}\mathfrak{g}^{R^+}_{\tau,>0} \subseteq \mathfrak{g}^{R^+}_{\tau,>0}$ is preserved by $\underline{\ell}^1_{\alpha_{\tau}}$.
\end{lemma}
\begin{proof}
    Let $x$ be a $1$-periodic orbit such that $\mathscr{P}(x)>0$. We have
    \begin{equation}
        \underline{\ell}^1_{\alpha_{\tau}}(x) = \sum_{d\geq 0} \frac{1}{d!}\underline{\ell}^{1+d}(x,\alpha_{\tau},\dotsi,\alpha_{\tau}).
    \end{equation}
    If $x_0$ is an orbit which contributes to the term $\underline{\ell}^{1+d}(x,\alpha_{\tau},\dotsi,\alpha_{\tau})$, then
    \begin{align*}
        C\cdot n_{x_0} + \kappa^{-1}A(x_0)
        &= \mathscr{P}(x_0) + \abs{x_0}\\
        &\geq \mathscr{P}(x) - 2d + 2(1+d) - 3 + \abs{x}+1\\
        &= C\cdot n_x + \kappa^{-1}A(x).
    \end{align*}
\end{proof}

For a graded complex $C^{\bullet}$, we define its degree-truncation at $p\in\mathbb{R}$, 
\begin{equation}
    \sigma_{<p} C^{\bullet} = \bigoplus_{i<p}C^i.
\end{equation}
It is equivalent to the quotient complex $C^{\bullet}/C^{\bullet\geq p}$ and has the property that
\begin{equation}\label{eq: cohomology isom}
    H^{j}(\sigma_{<p} C^{\bullet}) = H^{j}(C^{\bullet}) \quad\text{for all} \ j<p-1.
\end{equation}

\begin{lemma}\label{lem: qis}
    We have an isomorphism
\begin{equation}\label{eq: from alpha_0 to alpha_1}
    H^j(\sigma_{<p-1}\mathcal{F}^X_{\geq p}\wotimes R^+,\ell^1_{\alpha_1}) \cong H^j(\sigma_{<p-1}\mathcal{F}^X_{\geq p}\wotimes R^+,\ell^1_{\alpha_0})
\end{equation}
for all $p\in \mathbb{R}$. 
\end{lemma}
\begin{proof}
Using the $L_{\infty}$ relations on $\mathfrak{g}_{>1}$, we deduce that $\underline{\ell}^1_{\alpha_{\tau}}$ is a differential when restricted to $\sigma_{<p-1}(\mathcal{F}^X_{\geq p}\wotimes R^+)\wotimes \bk[\tau,d\tau]$. As a consequence, the result follows by the same argument as in the proof of Lemma \ref{lemma: qis}. 
\end{proof}

\begin{lemma}\label{lemma: Iso with SC}
    For any $j<p-2$, we have an isomorphism
    \begin{equation}
        H^j(\sigma_{<p-1}\mathcal{F}^X_{\geq p}\wotimes R^+,\ell^1_{\alpha_1}) \cong H^j(SC^*(X_{\sigma})\wotimes R^+,\ell_{\alpha_1}^1).
    \end{equation}
\end{lemma}
\begin{proof}
    Using the isomorphism in \eqref{eq: cohomology isom}, it is enough to show that the inclusion
    \begin{equation}
        (\mathcal{F}^X_{\geq p} \wotimes R^+, \ell^1_{\alpha_1})\rightarrow (SC^*(X_{\sigma})\wotimes R^+,\ell_{\alpha_1}^1)
    \end{equation}
    is a quasi-isomorphism. This is true by a spectral sequence comparison argument similar to the proof of Lemma \ref{lemma: qis}. Indeed, using the $Q$-filtration on both sides, together with the fact that $\alpha_1 \in Q_{\geq 1}\mathfrak{g}_{>1}^{R^+}$, the induced map on the $E_1$-page is the inclusion
    \begin{equation}
        (\mathcal{F}^X_{\geq p} \otimes Q_{\geq j} R^+/Q_{\geq j+1} R^+, \ell^1)\rightarrow(SC^*(X_{\sigma})\otimes Q_{\geq j} R^+/Q_{\geq j+1} R^+, \ell^1),
    \end{equation}
    which is an isomorphism by \cite[Proposition 1.14]{BSV}.
\end{proof}

\subsection{Recovering Quantum Cohomology}

In this section we complete the proof of Theorem \ref{thm}. 

For this we need to pass to the localization $R$ of $R^+$ at the product $q_1\dotsi q_N$,
\begin{equation}
    R  = \bk[q_1^{\pm 1},\dots,q_N^{\pm 1}].
\end{equation}
It has the structure of an $R^+$-algebra by the localization homomorphism, and of a $\Lambda^{\gr}_{\omega}$-algebra by the ring homomorphism
\begin{equation}
    \Lambda^{\gr}_{\omega} \rightarrow R,\quad e^A\mapsto \bq^{A\cdot \bD} \defeq q_1^{A\cdot D_1}\dotsi q_N^{A\cdot D_N}.
\end{equation}
Note that the filtrations on $R$ and $\Lambda^{\gr}_{\omega}$ are related,
\begin{equation}
    A_{\omega}(e^A) = 2\kappa \abs{\bq^{A\cdot \bD}}.
\end{equation}

By \cite[Proposition 1.15]{BSV}, Hypothesis \ref{hyp} implies that $\mathfrak{g}_{>0}$ is concentrated in non-negative degrees, and therefore that
\begin{equation}
    SC^*(X_{\sigma})\wotimes R = SC^*(X_{\sigma})\otimes R. 
\end{equation}

As $R$ is a flat $R^+$-module, the isomorphisms of Lemmas \ref{lem: qis} and \ref{lemma: Iso with SC} hold true over $R$:
\begin{align}
    H^j(\sigma_{<p-1}\mathcal{F}^X_{\geq p}\otimes R, \ell^1_{\alpha_1}) &\cong H^j(\sigma_{<p-1}\mathcal{F}^X_{\geq p}\otimes R, \ell^1_{\alpha_0}).\label{eq: alpha_0 to alpha_1 over R}\\
    H^j(\sigma_{<p-1}\mathcal{F}^X_{\geq p}\otimes R,\ell^1_{\alpha_1}) &\cong H^j(SC^*(X_{\sigma})\otimes R,\ell_{\alpha_1}^1).\label{eq: from F^X to SC}
\end{align}

Consider the telescoping complex
\begin{equation}
    QC^*(M;R) = \bigoplus_{n=1}^{\infty} CF^*(M,H_n;R)[t].
\end{equation}
It is equipped with a differential $\tilde{d}^{QC}$ which counts Floer trajectories \eqref{eq: Floer differential} and continuation maps using a generic $\omega$-compatible almost complex structure $(\tilde{J}_{n})_n$, and perturbation data for the continuation maps, on $M$. We emphasize that we require \emph{all} Floer trajectories and continuation maps for $(\tilde{J}_{n})_n$, and the corresponding perturbation data for the continuation maps, to be Fredholm regular. This is in contrast to the family $(J_n)_n$ and perturbation data from \S \ref{subsec: Almost complex structures}, for which we do not have regularity of Floer trajectories and continuation maps between $D$-orbits. Nonetheless, counts using the family $(J_n)_n$ can be used to define a diffential $d^{QC}$ on the complexes $\sigma_{<p-1}\mathcal{F}^M_{p}$, where
\begin{equation}
    \mathcal{F}^M_{p} = \mathcal{F}_{\geq p} QC^*(M;R),\quad p\in \mathbb{R}.
\end{equation}
\begin{lemma}\label{lemma: dQC from bad J}
    For any $j<p-2$, we have an isomorphism
    \begin{equation}
        H^j(\sigma_{<p-1}\mathcal{F}^M_{p},d^{QC}) \cong H^j(\sigma_{<p-1}\mathcal{F}^M_{p},\tilde{d}^{QC}).
    \end{equation}
\end{lemma}
\begin{proof}
    This follows from the usual TQFT style argument explained in \cite[\S 8]{PL-theory}, where one constructs a chain map
    \begin{equation} \label{eq: chain map}
        (\sigma_{<p-1}\mathcal{F}^M_{p},d^{QC}) \rightarrow  (\sigma_{<p-1}\mathcal{F}^M_{p},\tilde{d}^{QC})
    \end{equation}
    from a generic homotopy $J_{s,t}$ between $J_t$ and $\tilde{J}_t$. The difficulty in our case is that we do not have transversality for Floer trajectories between $D$-orbits. Instead, the compactness results needed to define \eqref{eq: chain map} and to show it is a chain map are Lemma \ref{lemma: no D-orbits} and Lemma \ref{lemma: no sphere bubbles}.
\end{proof}
 
The complex $(\sigma_{<p-1}\mathcal{F}^M_{p},d^{QC})$ only involves $1$-periodic orbits in $X_{\sigma}$, but unlike $\ell^1=d^{SC}$, the differential $d^{QC}$ counts Floer trajectories that interesect the divisor $D$. 

Define a morphism of $R$-modules
\begin{equation}
    \sigma_{<p-1}\mathcal{F}^X_{\geq p}\otimes_{\bk} R \rightarrow \sigma_{<p-1}\mathcal{F}^M_{p},\quad
    \Phi(\gamma) = [\gamma,u]\otimes q_1^{u\cdot D_1}\dotsi q_N^{u\cdot D_N}.
\end{equation}
It is not difficult to check that the image $\Phi(\gamma)$ does not depend on the choice of the cap $u$ for $\gamma$. For simplicity of notation, we write $\Phi(\gamma) = [\gamma,u]\otimes \bq^{u\cdot \bD}$.

\begin{lemma}\label{lemma: alpha_zero gives dQC}
    The map $\Phi$ is an isomorphism of chain complexes
    \begin{equation}
        (\sigma_{<p-1}\mathcal{F}^X_{\geq p} \otimes R,\ell^1_{\alpha_0})\rightarrow (\sigma_{<p-1}\mathcal{F}^M_{p}, d^{QC}).
    \end{equation}
\end{lemma}
\begin{proof}
    By definition, both chain complexes are built from $1$-periodic orbits in $X_{\sigma}$, so $\Phi$ is a vector space isomorphism. It remains to check that for any $1$-periodic orbit $\gamma$ in $X_{\sigma}$, we have
    \begin{equation}
        \Phi(\ell^1_{\alpha_0}(\gamma_1)) = d^{QC}(\Phi(\gamma_1)).
    \end{equation}
    We explain this in the simpler case when $N=1$; the general case is identical.

    Let $u_1$ be a cap for $\gamma_1$, and set $x_1 = (\gamma_1,u_1)$. Recall that
    \begin{equation}
        d^{QC}(x_1) = \sum_{x_0\neq x_1} \# \mapsms(x_0,x_1) x_0.
    \end{equation}
    On the other hand, 
    \begin{equation}
        \ell^1_{\alpha_0}(\gamma_1) = \sum_{\gamma_0}\sum_{d\geq 0} \frac{q^d}{d!}\langle\ell^{1+d}(\gamma_1,x^D,\dots,x^D) ,\gamma_0\rangle.
    \end{equation}
    To relate the two differentials, notice that by our \textbf{(Forgetfulness)} assumption on the choice of perturbation data,  
    for any cap $u_1$ for $\gamma_1$ we have a tautological map
    \begin{equation}\label{eq: map of moduli spaces}
        \mapsms^{\al}_{1+d,\bp,\bn}(\gamma_0,\gamma_1,x^D,\dots,x^D) \rightarrow \mapsms([\gamma_0,u_0],[\gamma_1,u_1]),
    \end{equation}
    where $\bp$ is empty, $\bn=(n_0,n_1,0,\dots,0)$, and $u_0$ is the unique cap for $\gamma_0$ such that $(u_1 - u_0) \cdot D = d$. 
    Our regularity assumptions on the perturbation data (Lemma \ref{lemma: transversality}) ensure that over the $0$-dimensional components of the moduli spaces, the map \eqref{eq: map of moduli spaces} is $d!$-to-1, with the fibre over $(r,u)$ in bijection with the set of labellings of the $d$ intersection points of the curve $u$ with $D$ by the labels $\{2,\ldots,d+1\}$. It follows that
    \begin{equation}
        \frac{1}{d!}\langle\ell^{1+d}(\gamma_1,x^D,\dots,x^D),\gamma_0\rangle = \langle d^{QC}x_1,x_0\rangle, 
    \end{equation}
    where $x_i = [\gamma_i,u_i]$. 
    Hence
    \begin{align*}
        \langle d^{QC}(x_1\otimes q^{u_1\cdot D}),x_0 \rangle 
        &= \frac{q^d}{d!}\langle\ell^{1+d}(\gamma_1,x_D,\dots,x_D),\gamma_0\rangle \otimes q^{-d+u_1\cdot D}\\
        &= \frac{q^d}{d!}\langle\ell^{1+d}(\gamma_1,x_D,\dots,x_D),\gamma_0\rangle \otimes q^{u_0\cdot D}\\
        &= \langle\Phi(\ell^1_{\alpha_0}(\gamma_1)),x_0\rangle
    \end{align*}
    as required.
\end{proof}

By combining the isomorphisms in \eqref{eq: alpha_0 to alpha_1 over R}, \eqref{eq: from F^X to SC}, together with the isomorphisms in Lemma \ref{lemma: dQC from bad J} and Lemma \ref{lemma: alpha_zero gives dQC}, we obtain our main result:
\begin{theorem}
    We have an isomorphism of $R$-modules
    \begin{equation}
        H^*(SC^*(X_{\sigma})\otimes R,\ell^1_{\alpha_1}) \cong QH^*(M,R).
    \end{equation}
\end{theorem}

\bibliographystyle{alpha}
\bibliography{Bibliography.bib}
\end{document}